%% file: DAHAs_and_Character_Varieties.tex
\documentclass[11pt]{article}
\input{Preamble}

\title{DAHAs of Type $C^\vee C_n$ and Character Varieties}
\author{Oleg Chalykh}
\author{Bradley Ryan}
\affil{}
\date{\sffamily\today}

\begin{document}

\maketitle\thispagestyle{empty}

\begin{abstract}
\noindent This paper studies the spherical subalgebra of the double affine Hecke algebra of type $C^\vee C_n$ and relates it, at the classical level $q = 1$, to a certain character variety of the four-punctured Riemann sphere. This establishes a conjecture from \cite{etingofganoblomkov}. As a by-product, we find a completed phase space for the trigonometric van Diejen system, explicitly integrate its dynamics and explain how it can be obtained via Hamiltonian reduction.
\end{abstract}

\section{Introduction}\label{sec: intro}

The double affine Hecke algebras (DAHAs) were invented by I. Cherednik who famously used them to prove Macdonald's conjectures \cite{cherednikJan95,cherednikDec95}. Their rational version, known as Cherednik algebras, was introduced and studied in the seminal paper by P. Etingof and V. Ginzburg \cite{etingofginzburg}. Since then, Cherednik algebras and DAHAs have found a wide range of mathematical uses and by now there is vast literature devoted to these classes of algebras.

One insight of \cite{etingofginzburg} was that the spherical subalgebra of the Cherednik algebra for a Coxeter group $W \subseteq \GL(V)$ provides an interesting Poisson deformation of the orbifold $T^\ast V/W$. In particular, for the symmetric group $W = S_n$ \cite{etingofginzburg} related the spherical subalgebra, at the classical level $t = 0$, to the Calogero-Moser space
\begin{equation*}
\mathcal{C}_n = \{X,Y \in \Mat_{n \times n}(\mathbb{C}) : \rank([X,Y] - \Id_n) = 1\}\GIT\GL_n(\mathbb{C}),
\end{equation*}
previously studied by G. Wilson \cite{wilson}. The spaces $\mathcal C_n$ have their origin in the theory of Hamiltonian reduction and integrable systems \cite{kazhdankostantsternberg}. These are affine smooth varieties (of dimension $2n$) that can be viewed as completed phase spaces for the classical Calogero-Moser particle system. In order to formulate this result of \cite{etingofginzburg}, we recall that the Cherednik algebra $\mathsf{H}_{t,c}$ of $W$ contains the group algebra $\mathbb{C}W$, and the spherical subalgebra of $\mathsf{H}_{t,c}$ is defined as $e\mathsf{H}_{t,c}e$, where $e$ is the group symmetriser. Recall also that the spherical subalgebra becomes commutative at the classical level $t = 0$.

\begin{theorem}[{cf. \cite[Theorem 1.23]{etingofginzburg}}]\label{thrm: EG isomorphism}
For $W = S_{n}$ and $c \neq 0$, we have an isomorphism 
\begin{equation*}
e\mathsf{H}_{0,c}e \cong \mathbb{C}[\mathcal{C}_n]
\end{equation*}
between the spherical subalgebra and the algebra of regular functions on the Calogero-Moser space. Equivalently, $\Spec(e\mathsf{H}_{0,c}e)\cong \mathcal{C}_n$ as affine algebraic varieties. 
\end{theorem}
This has the following representation-theoretic implication.
\begin{theorem}[{\cite[Theorem 1.24]{etingofginzburg}}]\label{thrm: EG representation theory}
For $W = S_n$, every finite-dimensional irreducible representation of $\mathsf{H}_{0,c}$ has dimension $n!$, and such representations are in 1-1 correspondence with points of $\mathcal C_n$, i.e.
\begin{equation*}
\Irrep(\mathsf{H}_{0,c}) \cong \mathcal{C}_n.
\end{equation*}
\end{theorem}
This motivates studying the varieties $\Spec(e\mathsf{H}_{0,c}e)$ in other situations. For example, Etingof and Ginzburg consider symplectic reflection groups $W = S_n \ltimes G^n$ for finite subgroups $G \subseteq \SU(2)$ and prove a generalisation of Theorem \ref{thrm: EG isomorphism} for the corresponding symplectic reflection algebras, see \cite[Theorem 11.16]{etingofginzburg}. 

Another result of that kind, closely related to the subject of the present paper, was obtained by A. Oblomkov in \cite{oblomkovJun04}, where he studied the analogous problem for the DAHA $H_{q,\tau}$ of type $\GL_n$. Recall that this DAHA contains the finite Hecke algebra of $S_n$, so one defines its spherical subalgebra as $e_\tau H_{q,\tau}e_\tau$, where $e_\tau$ is the Hecke symmetriser. Again, the spherical subalgebra becomes commutative at the classical level $q = 1$ and we have the following result.

\begin{theorem}[{\cite{oblomkovJun04}}]\label{thrm: Oblomkov isomorphism}
For $\tau$ not a root of unity, we have an isomorphism
\begin{equation*}
\Spec(e_\tau H_{1,\tau}e_\tau) \cong \text{CM}_\tau,
\end{equation*}
where
\begin{equation*}
CM_\tau = \{X,Y \in \GL_n(\mathbb{C}) : \rank(\tau XYX\inv Y\inv - \tau\inv\Id_n) = 1\}\GIT\GL_n(\mathbb{C}).
\end{equation*}
\end{theorem}
The space $CM_\tau$ was previously identified as a completed phase space for the Ruijsenaars--Schneider system \cite{fockrosly}. Oblomkov also establishes an analogue to Theorem \ref{thrm: EG representation theory}: finite-dimensional irreducible representations of $H_{1,\tau}$ are in 1-1 correspondence with points of $CM_\tau$. 

A further generalisation was proposed in \cite{etingofoblomkovrains,etingofganoblomkov} by way of introducing generalised DAHAs associated to the star-shaped Dynkin quivers of type $\widetilde{D}_4, \widetilde{E}_6, \widetilde{E}_7, \widetilde{E}_8$. In these cases, \cite[Conjecture 5.1.1]{etingofganoblomkov} states that the corresponding spherical subalgebras at the classical level $q = 1$ are isomorphic to certain character varieties of a punctured Riemann sphere. While some important steps towards a proof of this conjecture have been made in \cite{etingofganoblomkov}, it remained open so far.

Our main result is a proof of the aforementioned conjecture for the generalised DAHA of type $\widetilde{D}_4$. This generalised DAHA is isomorphic to the usual DAHA of type $C^\vee C_n$ which makes it amenable to the methods of \cite{oblomkovJun04}. To state the result, we need a replacement for the above space $CM_\tau$. Let $C_i=[\Lambda_i] \subseteq \GL_{2n}(\mathbb{C})$ for $i = 1,2,3,4$ denote the conjugacy classes represented by the matrices
\begin{equation}\label{eq: our eigendata}
\begin{aligned}
\Lambda_1 &= \diag(\smallunderbrace{-k_0\inv,\dots,-k_0\inv}_n,\smallunderbrace{k_0,\dots,k_0}_n),\\
\Lambda_2 &= \diag(\smallunderbrace{-u_0\inv,\dots,-u_0\inv}_n,\smallunderbrace{u_0,\dots,u_0}_n),\\
\Lambda_3 &= \diag(\smallunderbrace{-u_n\inv,\dots,-u_n\inv}_n,\smallunderbrace{u_n,\dots,u_n}_n),\\
\Lambda_4 &= \diag(\smallunderbrace{-k_n\inv,\dots,-k_n\inv}_n,\smallunderbrace{k_nt^{-2},\dots,k_nt^{-2}}_{n-1},\smallunderbrace{k_nt^{2n-2}}_1),
\end{aligned}
\end{equation}
where $k_0,k_n,t,u_0,u_n \in \mathbb{C}^\ast$ are parameters. We define an analogue of the varieties $\mathcal C_n$ and $CM_\tau$:
\begin{equation}\label{eq: CM space}
\mathcal{M}_{n} \coloneqq \{A_i \in C_i : A_1A_2A_3A_4 = \Id_{2n}\}\GIT\GL_{2n}(\mathbb{C}).
\end{equation}
This is a $\GL_{2n}(\mathbb{C})$-character variety of the Riemann sphere with four punctures, with the conjugacy class $C_i=[\Lambda_i]$ attached to the punctures (see Definition \ref{def: character variety} below). For generic parameters, $\mathcal{M}_{n}$ is a smooth affine variety of dimension $2n$. Our main result may now be stated.

\begin{theorem}\label{thrm: our isomorphism}
Let $\mathcal{H}_{q,\bm{\tau}}$ denote the DAHA of type $C^\vee C_n$ with parameters $q$, $\bm{\tau} = (k_0,k_n,t,u_0,u_n)$, and $e_\tau\mathcal{H}_{q,\bm{\tau}}e_\tau$ the corresponding spherical subalgebra. Then, for $q = 1$ and generic parameters $\bm{\tau}$, we have an isomorphism
\begin{equation}\label{eq: main result}
e_\tau\mathcal{H}_{1,\bm{\tau}}e_\tau \cong \mathbb{C}[\mathcal{M}_{n}].
\end{equation}
\end{theorem}

We also show that the finite-dimensional irreducible representations of $\mathcal{H}_{1,\bm{\tau}}$ are in 1-1 correspondence with points of $\mathcal{M}_n$, and suggest a quantisation of the isomorphism \eqref{eq: main result}.

Our proof of Theorem \ref{thrm: our isomorphism} follows the general strategy of \cite{oblomkovJun04}, albeit this requires a few additional ingredients. Note that it relies crucially on the existence of the so-called Basic Representation for the DAHA. Since such a representation is not available for the generalised DAHAs of type $\widetilde{E}_6$, $\widetilde{E}_7$, $\widetilde{E}_8$, for those cases our approach does not work and the conjecture \cite[Conjecture 5.1.1]{etingofganoblomkov} remains open. It should be mentioned that for $n = 1$, the result of Theorem \ref{thrm: our isomorphism} follows from \cite{oblomkovJan04}, where the spherical subalgebra of the DAHA of $C^\vee C_1$ type is identified with the ring of functions on an affine cubic surface.

The isomorphism \eqref{eq: main result} has a nice application in the theory of integrable systems. Our main result here (see Theorem \ref{thrm: dynamics on CM space in the second chart} below) says that $\mathcal{M}_{n}$ can be viewed as a completed phase space for the trigonometric van Diejen system. This particle system in the quantum setting is intimately related to Koornwinder polynomials. As we demonstrate, the corresponding classical Hamiltonian dynamics on $\mathcal{M}_{n}$ can be explicitly integrated. We also interpret it via Hamiltonian reduction using the Fock--Rosly framework \cite{fockrosly}.

Let us finish the introduction by mentioning some other contexts where the varieties $\mathcal{M}_n$, or closely-related spaces, appear. First, the varieties $\mathcal{M}_n$ can be viewed as monodromy spaces for the Fuchsian systems studied in \cite{kawakami,bertolacafassorubtsov} in the context of isomonodromy deformations and matrix Painlev\'e VI systems (for $n = 1$, it is the affine cubic surface well known in the Painlev\'e VI theory). Second, as a character variety, $\mathcal{M}_n$ has two other avatars: the spaces $\mathcal{M}_\text{dR}$ and $\mathcal{M}_\text{Dol}$ of (weakly) parabolic local systems and Higgs bundles, respectively. By the result of \cite{groechenig}, all three are diffeomorphic to the Hilbert scheme $\Hilb^n(S)$ for a certain surface $S$. Note that $\mathcal{M}_\text{Dol}$ hosts the elliptic Inozemtsev integrable system and is related to the Seiberg--Witten geometry of some $4d$ $\mathcal{N} = 2$ superconformal field theories \cite{argyreschalykhlu21}. Finally, from a more general perspective, there is a geometric approach to the moduli spaces of sheaves on non-commutative deformations of surfaces developed by E. Rains \cite{rains16,rains19}, and also his theory of elliptic DAHAs \cite{rains20} within which the DAHA studied here can be obtained in a trigonometric limit. This gives a moral reason, if not a proof, for the isomorphism \eqref{eq: main result}, by viewing both sides as two kinds of deformations of a particular Hilbert scheme.

\subsection*{Structure of the Paper}

In Section \ref{sec: DAHA}, we recall the main facts about the DAHA of type $C^\vee C_n$. In Section \ref{sec: character varieties}, we discuss character varieties of punctured Riemann surfaces and their basic properties. In Section \ref{sec: irreps}, we analyse a map $\Phi$, constructed in \cite{etingofganoblomkov}, that associates points of $\mathcal{M}_n$ to finite-dimensional representations of $\mathcal{H}_{1,\bm{\tau}}$. In Section \ref{sec: coordinates}, we introduce coordinates on $\mathcal{M}_n$ and show that $\Phi$ restricts to an isomorphism on a suitable coordinate chart. This puts all the ingredients in place to run the rest of the arguments as in \cite{oblomkovJun04}. In Section \ref{sec: isomorphism}, we recall the main steps in \textit{op. cit.}, which then leads us to the isomorphism in Theorem \ref{thrm: our isomorphism} (see Theorem \ref{thrm: main result}). We also suggest a possible way to quantise this isomorphism (see Proposition \ref{prop: quantised main result}). Finally, in Section \ref{sec: dynamics}, we apply our results to the trigonometric van Diejen system: we explicitly describe its dynamics on $\mathcal{M}_n$ and explain how it can be obtained by Hamiltonian reduction.

\subsection*{Acknowledgements}

We are grateful to Yu. Berest, P. Etingof, E. Rains and N. Reshetikhin for stimulating discussions and useful remarks. B. Ryan is supported by a PhD studentship from the UKRI (grant number EP/2434194).

\section{The DAHA \texorpdfstring{of Type $C^\vee C_n$}{}}
\label{sec: DAHA}

The double affine Hecke algebras (DAHAs) were introduced by Cherednik \cite{cherednik92} for reduced affine root systems. Their definition for the non-reduced affine root system of type $C^\vee C_n$ was suggested by Sahi \cite{sahi} and is recalled next. 

\subsection{The DAHA}\label{subsec: DAHA}

\begin{definition}\label{def: DAHA of rank n}
Let $k_0, k_n, t, u_0, u_n, q^{1/2} \in \mathbb{C}^\ast$ and write $\bm{\tau}$ for the parameters $(k_0,k_n,t,u_0,u_n)$. The \defn{double affine Hecke algebra of type $C^\vee C_n$} is the associative algebra $\mathcal{H}_{q,\bm{\tau}}$ over $\mathbb{C}[\bm{\tau}^{\pm1},q^{\pm1/2}]$ generated by the elements $T_0^{\pm1},\dots,T_n^{\pm1}$ and $X_1^{\pm1},\dots,X_n^{\pm1}$, subject to the following relations:
\begin{enumerate}[label=\rom]
	\item $T_0 - T_0\inv = k_0 - k_0\inv$. 
	\item $T_i - T_i\inv = t - t\inv$. \tabto{8cm} ($i=1,\dots,n-1$)
	\item $T_n - T_n\inv = k_n - k_n\inv$. 
	\item $[T_i, T_j] = 0$. \tabto{8cm} ($\abs{i - j} > 1$)
	\item $T_0T_1T_0T_1 = T_1T_0T_1T_0$.
	\item $T_iT_{i+1}T_i = T_{i+1}T_iT_{i+1}$. \tabto{8cm} ($i=1, \dots, n-2$)
	\item $T_{n-1}T_nT_{n-1}T_n = T_nT_{n-1}T_nT_{n-1}$.
	\item $[X_i, X_j] = 0$. \tabto{8cm} ($1 \leq i < j \leq n$)
	\item $[T_i, X_j] = 0$. \tabto{8cm} ($j \neq i,i+1$)
	\item $T_iX_i = X_{i+1}T_i\inv$. \tabto{8cm} ($i=1, \dots, n-1$)
	\item $T_0^\vee - (T_0^\vee)\inv = u_0 - u_0\inv$. \tabto{8cm} ($T_0^\vee \coloneqq q^{-1/2} T_0\inv X_1$)
	\item $T_n^\vee - (T_n^\vee)\inv = u_n - u_n\inv$. \tabto{8cm} ($T_n^\vee \coloneqq X_n\inv T_n\inv$)
\end{enumerate}
\end{definition}

The elements $T_1, \dots, T_n$ generate a subalgebra $H_n \subseteq \mathcal{H}_{q,\bm{\tau}}$ isomorphic to the finite Hecke algebra of type $C_n$. For $k_n = t = 1$, it becomes isomorphic to the group algebra $\mathbb{C}W$ of the Weyl group of type $C_n$, generated by $s_1,\dots,s_n$ subject to $s_i^2 = 1$ and the above relations (iv), (vi), (vii). If $w = s_{i_1}\cdots s_{i_\ell}$ is a reduced decomposition of $w \in W$, we write 
\begin{equation*}
T_w \coloneqq T_{i_1}\cdots T_{i_\ell}.
\end{equation*}
It is well known that $T_w$ does not depend on the choice of a reduced decomposition of $w$. As a vector space, $H_n$ is generated by the elements $T_w$ with $w \in W$. The mapping $T_i \mapsto t$ ($i=1,\dots,n-1$), $T_n\mapsto k_n$ extends to a one-dimensional representation $\chi : T_w \mapsto \tau_w$ of $H_n$. With this representation, one associates the following idempotent, called the \textit{Hecke symmetriser}, in $H_n$ (which becomes the usual group symmetriser in $\mathbb{C}W$ when $k_n = t = 1$): 
\begin{equation}\label{eq: Hecke symmetriser}
e_\tau \coloneqq \frac{1}{\sum_{w \in W}\tau_w^2}\sum_{w \in W} \tau_wT_w.
\end{equation}

\begin{definition}\label{def: spherical subalgebra}
The \defn{spherical subalgebra} of the DAHA $\mathcal{H}_{q,\bm{\tau}}$ is $e_\tau\mathcal{H}_{q,\bm{\tau}}e_\tau$. 
\end{definition}

\begin{notation}\label{not: S and S dagger}
As a way to alleviate notation here on in, we shall introduce the following:
\begin{equation}\label{eq: S and S dagger}
S \coloneqq T_1 \cdots T_{n-1} \qquad\text{and}\qquad S^\dagger \coloneqq T_{n-1} \cdots T_1.
\end{equation}
\end{notation}

The following corollary of the DAHA relations will be important later. 
 
\begin{lemma}\label{lem: product of DAHA elements is the identity}
We have the relation $q^{1/2}T_0T_0^\vee ST_n^\vee T_nS^\dagger = 1$. 
\end{lemma}

\begin{proof}
Using $T_0^\vee = q^{-1/2}T_0\inv X_1$ and $T_n^\vee = X_n\inv T_n\inv$, the needed relation is $X_1SX_n\inv S^\dagger = 1$, rather $S^\dagger X_1 = X_nS\inv$ which is an easy consequence of relations (x) of the DAHA.
\end{proof}

\subsection{Duality}\label{subsec: duality}

Apart from the obvious commutative subalgebra $\mathbb{C}[\bm{X}^{\pm1}]$ of Laurent polynomials in $X_1, \dots, X_n$, $\mathcal{H}_{q,\bm{\tau}}$ contains another commutative Laurent polynomial subalgebra $\mathbb{C}[\bm{Y}^{\pm1}]$ in $Y_1, \dots, Y_n$, where
\begin{equation}\label{yi}
Y_i \coloneqq T_i \cdots T_{n-1}T_nT_{n-1}\cdots T_1T_0T_1\inv \cdots T_{i-1}\inv.
\end{equation}
For $\lambda\in\mathbb{Z}^n$, let us write $\bm{X}^\lambda$ for $X_1^{\lambda_1}\dots X_n^{\lambda_n}$, and similarly for $\bm{Y}^\lambda$. The elements 
\begin{equation*}
\bm{X}^\lambda T_w\bm{Y}^\mu, \qquad \lambda,\mu\in\mathbb{Z}^n 
\end{equation*}
form a basis of $\mathcal{H}_{q,\bm{\tau}}$. In other words, every element $h\in \mathcal{H}_{q,\bm{\tau}}$ has a unique presentation as
\begin{equation*}
h=\sum_{\lambda,\mu\in\mathbb{Z}^n, w\in W}h_{\lambda, w, \mu} \bm{X}^\lambda T_w\bm{Y}^\mu, \qquad \text{with}\ h_{\lambda, w, \mu}\in\mathbb{C}.
\end{equation*}
This is referred to as the \textit{PBW property} of $\mathcal{H}_{q,\bm{\tau}}$, from which we have a vector space isomorphism
\begin{equation*}
\mathcal{H}_{q,\bm{\tau}} \cong \mathbb{C}[\bm{X}^{\pm1}] \otimes H_n \otimes \mathbb{C}[\bm{Y}^{\pm1}].
\end{equation*}
This indicates that $X_i$ and $Y_i$ should be viewed on the same footing (which is far from obvious from the definitions). This is manifested by the so-called \textit{duality}, an isomorphism between two DAHAs with different parameters as formulated in the following result.

\begin{theorem}[{\cite[Theorem 4.2]{sahi}}, Duality Isomorphism]\label{thrm: duality isomorphism}
Let $\mathcal{H}_{q,\bm{\tau}}$ be the DAHA with parameters $\bm{\tau} = (k_0,k_n,t,u_0,u_n)$, and $\mathcal{H}_{q\inv,\bm{\sigma}}$ the DAHA with parameters $\bm{\sigma} = (u_n\inv,k_n\inv,t\inv,u_0\inv,k_0\inv)$. There is a unique algebra isomorphism $\varepsilon : \mathcal{H}_{q,\bm{\tau}} \to \mathcal{H}_{q\inv,\bm{\sigma}}$ given on generators by
\begin{align*}
\varepsilon(T_0) &= T_1 \cdots T_{n-1}(T_n^\vee)\inv T_{n-1}\inv \cdots T_{1}\inv,\\
\varepsilon(T_i) &= T_i\inv,\\
\varepsilon(X_i) &= Y_i.
\end{align*}
It is equal to its inverse, with swapped $\bm{\tau} \leftrightarrow \bm{\sigma}$ and $q \leftrightarrow q^{-1}$. In particular, $\varepsilon(Y_i) = X_i$.
\end{theorem}

\subsection{The Basic Representation}\label{subsec: Basic Representation}

A crucial ingredient for what follows is the Basic Representation of $\mathcal{H}_{q,\bm{\tau}}$ \cite[Theorem 3.1]{sahi}, which extends the representation of the affine Hecke algebra of type $\widetilde C_n$ found in \cite{noumi}.

Consider the ring of $q$-difference operators in $n$ variables $X_1,\dots, X_n$. It is generated by $X_i^{\pm1}, P_i^{\pm1}$ with $i=1, \dots, n$, subject to the following relations for all indices $i$ and $j$:
\begin{equation}\label{eq: q-commutation relation}
[X_i,X_j] = [P_i,P_j]=0, \qquad P_iX_j = q^{\delta_{ij}}X_jP_i.
\end{equation}
We denote this ring $\mathbb{C}_q[\bm{X}^{\pm1},\bm{P}^{\pm1}]$; when $q=1$, it is the ring $\mathbb{C}[\bm{X}^{\pm1},\bm{P}^{\pm1}]$ of Laurent polynomials in $2n$ variables. Localising on the (Ore) set $\mathbb{C}(\bm{X})\setminus\{0\}$ of non-zero rational functions, we obtain the ring $\mathbb{C}_q(\bm{X})[\bm{P}^{\pm1}]$ of $q$-difference operators with rational coefficients, denoted $\mathscr{D}_q$ for brevity:
\begin{equation}\label{eq: Dq}
\mathscr{D}_q \coloneqq \mathbb{C}_q(\bm{X})[\bm{P}^{\pm1}].
\end{equation}
The Weyl group of type $C_n$ consisting of \textit{sign-changing} permutations on $n$ letters,
\begin{equation}\label{eq: W}
W = S_n \ltimes \mathbb{Z}_2^n,
\end{equation}
is generated by transpositions $\s_{ij}$ and sign-reversals $\s_i$. It acts on on $\mathbb{C}_q[\bm{X}^{\pm1},\bm{P}^{\pm1}]$ and $\mathscr{D}_q$ by 
\begin{equation}\label{eq: sign reversal elements}
\begin{alignedat}{2}
\s_i &: X_i \mapsto X_i\inv, \quad&& P_i \mapsto P_i\inv,\\
\s_{ij} &: X_i \mapsto X_j, \quad&& P_i \mapsto P_j,\\
\s_{ij}^+ = \s_{ij}\s_i\s_j &: X_i \mapsto X_j\inv, \quad&& P_i \mapsto P_j\inv.
\end{alignedat}
\end{equation}

We can, therefore, make the following definition.

\begin{definition}\label{def: difference-reflection operators}
The algebra of \defn{$q$-difference-reflection operators} is the semi-direct product 
\begin{equation*}
\mathscr{D}_q \rtimes \mathbb{C}W.
\end{equation*}
\end{definition}

Elements of $\mathscr{D}_q \rtimes \mathbb{C}W$ are finite linear combinations 
\begin{equation*}
\sum_{\lambda\in\mathbb{Z}^n, w\in W}a_{\lambda, w}\bm{P}^\lambda w, \qquad\text{with}\ a_{\lambda, w}\in\mathbb{C}(\bm{X}).
\end{equation*}
Now, introduce the notation
\begin{equation}\label{eq: si in terms of sij}
s_0 \coloneqq P_1^{-1}\s_1, \qquad
s_n \coloneqq \s_n, \qquad
s_i \coloneqq \s_{i\,i+1} \quad (i=1,\dots,n-1).
\end{equation}
They satisfy the relations of the affine Weyl group $\widetilde{W}$ of type $\widetilde{C}_n$, cf. Definition \ref{def: DAHA of rank n}(v)-(vii).

\begin{proposition}[{\cite[(13)]{sahi}}, Basic Representation]\label{prop: Basic Representation}
There is an injective homomorphism of algebras $\beta : \mathcal{H}_{q,\bm{\tau}} \to \mathscr{D}_q \rtimes \mathbb{C}W$ given on generators by $X_i \mapsto X_i$ and $T_i \mapsto \tau_i + c_i(\bm{X})(s_i - 1)$, for
\begin{align}
c_0(\bm{X}) &= k_0\inv\frac{(1 - q^{1/2}k_0u_0X_1\inv)(1 + q^{1/2}k_0u_0\inv X_1\inv)}{1 - qX_1^{-2}},\label{eq: c0(X)}\\
c_i(\bm{X}) &= \frac{t\inv - tX_iX_{i+1}\inv}{1 - X_iX_{i+1}\inv},\label{eq: ci(X)} &\quad(i=1,\dots,n-1)\\
c_n(\bm{X}) &= k_n\inv\frac{(1 - k_nu_nX_n)(1 + k_nu_n\inv X_n)}{1 - X_n^2},\label{eq: cn(X)}
\end{align}
where we use $\overline{\textover[c]{$\cdot$}{w}}$ to denote the difference between a parameter and its inverse, e.g. $\overline{t} = t - t\inv$.
\end{proposition}

\begin{remark}\label{rem: Basic Representation at the classical level}
Our $q$ corresponds to $q^2$ in \cite{sahi}, which furthermore is assumed not to be a root of unity. The fact that the same formulae define an injective homomorphism for all $q$, including the classical level $q = 1$, is stated in \cite{oblomkovJun04}; for a proof, see \cite[Theorem 3.1]{stokman} and \cite[Theorem 3.2]{vandiejenemsizzurrian}.
\end{remark}

\begin{remark}
For $q \neq 1$, elements of $\mathscr{D}_q \rtimes \mathbb{C}W$ (and hence elements of $\mathcal{H}_{q,\bm{\tau}}$) can be viewed as operators acting on functions of $n$ variables, with $P_i$'s acting as multiplicative shifts
\begin{equation}\label{eq: Pi-action on X variables}
P_i \act f(X_1,X_2,\dots,X_n) = f(X_1,\dots,X_{i-1},qX_i, X_{i+1},\dots,X_n)
\end{equation}
and with $s_i$ acting in the following way:
\begin{equation}\label{eq: Sn-action on X variables}
\begin{aligned}
s_0 \act f(X_1, X_2 \dots, X_n) &= f(qX_1\inv, X_2, \dots, X_n),\\
s_i \act f(X_1, X_2 \dots, X_n) &= f(X_1, \dots, X_{i-1}, X_{i+1}, X_i, X_{i+2}, \dots, X_n),\\
s_n \act f(X_1, X_2 \dots, X_n) &= f(X_1, \dots, X_{n-1}, X_n\inv).
\end{aligned}
\end{equation}
\end{remark}

One important corollary of Proposition \ref{prop: Basic Representation} is the following localisation property of $\mathcal{H}_{q,\bm{\tau}}$. Write $\mathcal{S}$ for the multiplicative subset generated by the following elements, for $i \neq j$ and $r \in \mathbb{Z}$:
\begin{equation}\label{eq: delta factors}
\begin{gathered}
1 - q^rX_i^{\pm1}X_{j}^{\pm1}, \qquad 1 - q^rt^2X_i^{\pm1}X_{j}^{\pm1}, \qquad 1 - q^{r}X_i^{\pm2},\\[2mm]
1 - q^{r/2}k_0u_0X_i^{\pm1}, \qquad 1 + q^{r/2}k_0u_0\inv X_i^{\pm1} \qquad 1 - q^{r/2}k_nu_nX_i^{\pm1}, \qquad 1 + q^{r/2}k_nu_n\inv X_i^{\pm1}.
\end{gathered}
\end{equation}
It is easy to verify this is both a right and left Ore localising set in $\mathcal{H}_{q,\bm{\tau}}$; view the DAHA as its image under the Basic Representation. The corresponding Ore localisation is denoted $\mathcal{H}_{q,\bm{\tau}}[\mathcal{S}\inv]$.

\begin{proposition}\label{rem: localisation of the DAHA by S}
The Basic Representation $\beta$ induces an isomorphism of the Ore localisations
\begin{equation*}
\mathcal{H}_{q,\bm{\tau}}[\mathcal{S}\inv] \cong \mathbb{C}_q[\bm{X}^{\pm1}, \bm{P}^{\pm1}][\mathcal{S}\inv]\rtimes \mathbb{C}W.
\end{equation*}
\end{proposition}

When $q = 1$, we can replace the Ore set $\mathcal{S}$ by the multiplicative set generated by
\begin{equation}\label{eq: delta}
\delta(\bm{X}) \coloneqq \prod_{i=1}^n(1 - X_i^2)(1 - X_i^{-2})\prod_{j\neq k}(1 - X_jX_k)(1 - X_j\inv X_k)(1 - X_jX_k\inv)(1 - X_j\inv X_k\inv)
\end{equation}
and
\begin{equation}\label{eq: delta tau}
\begin{multlined}
\delta_{\bm{\tau}}(\bm{X}) \coloneqq \prod_{i<j}(1 - t^2X_iX_j)(1 - t^2X_i\inv X_j)(1 - t^2X_iX_j\inv)(1 - t^2X_i\inv X_j\inv)\\
\times \prod_{i=1}^n(1 - k_0u_0X_i)(1 + k_0u_0\inv X_i)(1 - k_nu_nX_i)(1 + k_nu_n\inv X_i)\\
\qquad\qquad\times\prod_{i=1}^n(1 - k_0u_0X_i\inv)(1 + k_0u_0\inv X_i\inv)(1 - k_nu_nX_i\inv)(1 + k_nu_n\inv X_i\inv).
\end{multlined}
\end{equation}

In the sequel, we use subscript $\delta(\bm{X})\delta_{\bm{\tau}}(\bm{X})$ for localisation on the product of \eqref{eq: delta} and \eqref{eq: delta tau}.

\begin{corollary}\label{cor: localisation of the DAHA by delta}
Let $\mathcal{H} = \mathcal{H}_{1\bm{\tau}}$. The Basic Representation $\beta$ induces isomorphisms of localisations
\begin{equation*}
\mathcal{H}_{\delta(\bm{X})\delta_{\bm{\tau}}(\bm{X})} \cong \mathbb{C}[\bm{X}^{\pm1}, \bm{P}^{\pm1}]_{\delta(\bm{X})\delta_{\bm{\tau}}(\bm{X})} \rtimes \mathbb{C}W, \qquad e_\tau\mathcal{H}_{\delta(\bm{X})\delta_{\bm{\tau}}(\bm{X})}e_\tau \cong \mathbb{C}[\bm{X}^{\pm1}, \bm{P}^{\pm1}]_{\delta(\bm{X})\delta_{\bm{\tau}}(\bm{X})}^We_\tau.
\end{equation*}
\end{corollary}

\section{Character Varieties}\label{sec: character varieties}

We begin by recalling the character varieties of punctured Riemann surfaces and some of their fundamental properties that we will need.

\subsection{General Properties}\label{subsec: general properties}

Given $g, k\ge 0$ and conjugacy classes $C_1,\dots, C_k \subseteq \GL_n(\mathbb{C})$, consider
\begin{equation*}
\mathfrak{R}_{g,k} = \{X_1, Y_1, \ldots, X_g, Y_g \in \GL_n(\mathbb{C}), A_i \in C_i : X_1Y_1X_1\inv Y_1\inv\dots X_gY_gX_g\inv Y_g\inv A_1\cdots A_k = \Id_n\}.
\end{equation*}
There is a natural action of $\GL_n(\mathbb{C})$ on ${\mathfrak{R}}_{g,k}$ by conjugation, so we define the
\defn{$\GL_n$-character variety} $\mathfrak{M}_{g,k}$ as the GIT-quotient 
\begin{equation}\label{def: character variety}
\mathfrak{M}_{g,k} \coloneqq {\mathfrak{R}}_{g,k}\GIT\GL_n(\mathbb{C}).
\end{equation}
The affine variety $\mathfrak{M}_{g,k}$ can be viewed as the moduli space of (semi-simple) representations of the fundamental group of a genus $g$ Riemann surface $\Sigma_{g,k}$ with $k$ punctures, where we fix the conjugacy classes representing the loops around the punctures.

Let us assume that the conjugacy classes $C_i$ are semi-simple so they are represented by diagonal matrices (that encode the eigenvalues, possibly repeated, of $A_i$). Let us write $\lambda_{ij}$ with $i=1,\dots, k$ and $j=1,\dots, m_{i}$ for the distinct eigenvalues of $A_i$, and $\mu_{ij}$ for the multiplicity of the eigenvalue $\lambda_{ij}$. We have $\mu_{i1} + \dots + \mu_{i\,m_i}=n$ for all $i$.

\begin{definition}\label{def: generic eigendata}
Let $\lambda_{ij}$, $i=1,\dots, k$, $j=1,\dots, m_i$ denote the distinct eigenvalues of semi-simple $A_i\in\GL_n(\mathbb C)$, with multiplicities $\mu_{ij}$. The conjugacy classes represented by the $A_i$ are \defn{generic} if
\begin{equation*}
\prod_{i=1}^k \prod_{j=1}^{m_i}\lambda_{ij}^{\mu_{ij}} = 1
\end{equation*}
and, for any $1 \leq s < n$ and a collection of numbers $\nu_{ij} \leq \mu_{ij}$ with $\nu_{i1} + \cdots + \nu_{i\,m_i} = s$ for all $i$,
\begin{equation*}
\prod_{i=1}^k \prod_{j=1}^{m_i}\lambda_{ij}^{\nu_{ij}} \neq 1.
\end{equation*}
\end{definition}

The genericity condition guarantees that the $\GL_n(\mathbb C)/\mathbb C^*$-action on $\mathfrak R_{g,k}$ is free. We henceforth assume the $\mathfrak{M}_{g,k}$ always have semi-simple generic conjugacy classes in the sense of Definition \ref{def: generic eigendata}.

\begin{theorem}[{\cite[Theorem 2.1.5]{hauselletellierrodriguez-villegas11}}]\label{thrm: smooth character variety}
If non-empty, $\mathfrak{M}_{g,k}$ is a smooth equidimensional variety of dimension
\begin{equation}\label{eq: character variety dimension}
d = 2 + (2g + k - 2)n^2 - \sum_{i,j}\mu_{ij}^2.
\end{equation}
\end{theorem}

The next fundamental property of the character varieties is a much more difficult result.

\begin{theorem}[{\cite[Theorem 1.1.1]{hauselletellierrodriguez-villegas13}}]\label{eq: connected character variety}
If non-empty, $\mathfrak{M}_{g,k}$ is a connected variety.
\end{theorem}

\begin{remark}\label{rem: set of DS solutions}
In the case of genus $g = 0$ (a punctured Riemann sphere), the points of $\mathfrak{R}_{0,k}$ describe $k$-tuples of matrices in chosen conjugacy classes whose product is the identity. Describing $\mathfrak{R}_{0,k}$ (in particular, determining whether it is non-empty) is the \textit{multiplicative Deligne-Simpson problem}:
\begin{equation}\label{eq: DS problem}
\mathfrak{R}_{0,k} = \{A_i \in C_i : A_1\cdots A_k = \Id_n\}.
\end{equation}
\end{remark}

\subsection{Calogero--Moser Spaces of Type $C^\vee C_n$}\label{subsec: CM space}

It is clear now that the variety $\mathcal M_n$ \eqref{eq: CM space} is a $\GL_{2n}(\mathbb{C})$-character variety $\mathfrak{M}_{0,4}$ of the four-punctured Riemann sphere, with \eqref{eq: our eigendata} describing the conjugacy classes $C_i$ at the punctures.

\begin{definition}\label{def: our CM space}
The \defn{Calogero--Moser space of type $C^\vee C_n$} is the character variety \eqref{eq: CM space} of the four-punctured sphere with conjugacy classes $C_i = [\Lambda_i]$ at the punctures chosen according to \eqref{eq: our eigendata}.
\end{definition}

\begin{definition}\label{def: generic paramteters}
We call $\bm{\tau} = (k_0,k_n,t,u_0,u_n)$ \defn{generic} if, for every $a,b,c,d,e \in \mathbb{Z}$ not all-zero,
\begin{equation*}
k_0^ak_n^bt^cu_0^du_n^e \neq 1.
\end{equation*}
\end{definition}

\begin{theorem}\label{thrm: CM space is smooth and irreducible}
For generic $\bm{\tau}$, the space $\mathcal{M}_{n}$ is a smooth irreducible $2n$-dimensional affine variety.
\end{theorem}

\begin{proof}
The fact that $\mathcal{M}_n$ is non-empty will follow from an explicit construction of a coordinate chart in Section \ref{sec: coordinates}. The rest follows by applying Theorems \ref{thrm: smooth character variety} and \ref{eq: connected character variety}.
\end{proof}

\begin{remark}\label{rem: Oblomkov's CM space}
It is worth mentioning the space $CM_\tau$ considered by Oblomkov can be written as
\begin{equation*}
CM_\tau = \{X,Y \in \GL_n(\mathbb{C}) : XYX\inv Y\inv \in [\Lambda]\}\GIT\GL_n(\mathbb{C})
\end{equation*}
with $[\Lambda] \subseteq \GL_n(\mathbb{C})$ the conjugacy class represented by $\Lambda = \diag(\tau^{-2},\ldots,\tau^{-2},\tau^{2n-2})$. This is nothing but a $\GL_n$-character variety of a one-punctured torus. Hence, the use of Theorem \ref{eq: connected character variety} circumvents the rather technical proof of connectedness in \cite{oblomkovJun04}.
\end{remark}

\subsection{Alternative Form}\label{subsec: alternative form}

It is sometimes convenient to use the following description of the space $\mathcal{M}_n$. For $V = \mathbb{C}^{2n}$, let $\mathcal{R}_n$ be the set of $X,Y,T \in \GL(V)$ and $(v,w) \in \Hom(\mathbb{C},V)\oplus \Hom(V,\mathbb{C})$ subject to
\begin{align}
T - T\inv &= (u_0 - u_0\inv)\Id_V,\label{eq: T relation}\\
XT\inv - TX\inv &= (k_0 - k_0\inv)\Id_V,\label{eq: XT relation}\\
T\inv Y\inv - YT &=(u_n - u_n\inv)\Id_V,\label{eq: TY relation}\\
tYTX\inv - t\inv XT\inv Y\inv &= (k_nt\inv - k_n\inv t)\Id_V + (t - t\inv)vw,\label{eq: YTX relation}\\
wv &= \frac{t^{2n}-1}{t^2-1}k_n + \frac{1-t^{-2n}}{1-t^{-2}}k_n\inv.\label{eq: wv relation}
\end{align}
There is a natural action of $g \in \GL(V)$ on $\mathcal{R}_n$, namely
\begin{equation}\label{eq: alternative form action}
g \act (X,Y,T,v,w) \coloneqq (gXg\inv, gYg\inv, gTg\inv, gv, wg\inv).
\end{equation}

\begin{proposition}\label{prop: alternative form}
Let $\bm{\tau} = (k_0,k_n,t,u_0,u_n)$ be generic in the sense of {\normalfont Definition \ref{def: generic paramteters}}. Then, the action \eqref{eq: alternative form action} on $\mathcal{R}_n$ is free and we have an isomorphism of varieties $\mathcal{M}_n \cong \mathcal{R}_n/\GL(V)$.
\end{proposition}

\begin{proof}
To semi-simple matrices $A_1, A_2, A_3, A_4$ with the eigenvalues \eqref{eq: our eigendata}, associate $X = A_1A_2$, $Y = A_4A_1 = A_3\inv A_2\inv$ and $T = A_2$. Let $\lambda_1, \dots, \lambda_{2n}$ be the eigenvalues of $A_4$ in \textit{reverse} order as they appear in \eqref{eq: our eigendata}, with respective corresponding eigenvectors $v_1, \dots, v_{2n}$. We then take $v$ to be the map $\mathbb{C} \to V$ sending $1$ to $v_1$, and define the map $w : V \to \mathbb{C}$ by
\begin{equation*}
w(v_1) = \frac{t^{2n}-1}{t^2-1}k_n + \frac{1-t^{-2n}}{1-t^{-2}}k_n\inv, \qquad w(v_i) = 0 \ \text{for $i=2,\dots,n$}.
\end{equation*}
It is easy to see that such $X, Y, T$, $v, w$ satisfy the relations \eqref{eq: T relation}--\eqref{eq: wv relation}.

As for a map in the opposite direction, we take $X, Y, T$, $v, w$ satisfying the relations \eqref{eq: T relation}--\eqref{eq: wv relation} and define $A_1 = XT\inv$, $A_2 = T$, $A_3 = T\inv Y\inv$ and $A_4 = Y T X\inv$. Obviously, $A_1A_2A_3A_4 = \Id_V$. From \eqref{eq: XT relation}, $A_1$ is diagonalisable with eigenvalues from the set $\{k_0, -k_0\inv\}$; one can argue similarly for $A_2$ and $A_3$, but we still need to establish that the eigenvalues have equal multiplicity. On the other hand, \eqref{eq: YTX relation} and \eqref{eq: wv relation} will imply that $tA_4 - t\inv A_4\inv$ is diagonalisable with two eigenvalues: $k_nt\inv - k_n\inv t$ (of multiplicity $2n-1$) and $k_nt^{2n-1} - k_n\inv t^{-2n+1}$ (of multiplicity one). This implies $A_4$ is also diagonalisable, with $2n-1$ eigenvalues from the set $\{k_nt^{-2}, -k_n\inv\}$ and one from $\{k_nt^{2n-2}, -k_n\inv t^{-2n}\}$. But the product of the matrix determinants is one, which gives us a relation of the form $k_0^ak_n^bt^cu_0^du_n^e = 1$ with some $a, b, c, d, e \in \mathbb{Z}$. Since the parameters are generic, the only possibility is that $a = b = c = d = e = 0$. This forces $A_1, A_2, A_3$ to have two eigenvalues each of multiplicity $n$. This, in turn, implies that $\det(A_4) = (-1)^n$, and we find that this is only possible if $A_4$ has eigenvalues as in \eqref{eq: our eigendata}.
\end{proof}

\subsection{Representations of Quivers and Multiplicative Quiver Varieties}\label{subsec: quiver varieties}

In the genus $g = 0$ case, the character varieties $\mathfrak{M}_{0,k}$ of a punctured Riemann sphere can be interpreted in the context of representations of star-shaped quivers and multiplicative quiver varieties, following Crawley-Boevey and Shaw \cite{crawley-boevey,crawley-boeveyshaw}. This is briefly recalled below.

Throughout, our quivers $Q = (Q_0,Q_1)$ have vertices $Q_0$ and arrows $Q_1$. For any arrow $v \xrightarrow{a} w$, we call $h(a) = w$ the \textit{head} and $t(a) = v$ the \textit{tail}. A \textit{path} is a concatenation of arrows, read right-to-left (this agrees with the convention in \cite{crawley-boeveyshaw}). A \textit{quiver representation} is an assignment of vector spaces to the vertices and linear maps to the arrows. The \textit{dimension vector} $\mathbf{n} = (n_v)_{v \in Q_0}$ is the tuple of vector space dimensions, so the corresponding space of representations of $Q$ is
\begin{equation}\label{eq: Rep(Q)}
\Rep(Q,\mathbf{n}) \cong \prod_{a \in Q_1}\Mat_{n_{h(a)} \times n_{t(a)}}(\mathbb{C}).
\end{equation}

There is a natural action on the space \eqref{eq: Rep(Q)} by simultaneous conjugation, in other words by
\begin{equation}\label{eq: GL(n)}
\GL(\mathbf{n}) \coloneqq \prod_{v \in Q_0}\GL_{n_v}(\mathbb{C}).
\end{equation}
Of course, the action by $\{z\Id_n : z \neq 0\} \cong \mathbb{C}^\ast$ is trivial, so we can consider $\PGL(\mathbf{n}) \coloneqq \GL(\mathbf{n})/\mathbb{C}^\ast$.

The \textit{double} of $Q$ is $\overline{Q} = (Q_0,\overline{Q}_1)$ which has the same vertex set but an enlarged arrow set, where for every arrow $v \xrightarrow{a} w$ in $Q_1$, the reverse arrow $w \xrightarrow{a^\ast} v$ is adjoined. The \textit{path algebra} $\mathbb{C}Q$ is the algebra generated by the arrows $a \in Q_1$ and trivial paths $e_v$ at each $v \in Q_0$. We write $\mathbb{C}\overline{Q}[(1 + aa^\ast)\inv]$ for the path algebra of $\overline{Q}$ with formally inverted elements of the form $1 + aa^\ast$ for $a \in \overline{Q}_1$. For $a \in \overline{Q}_1$, we define $\varepsilon(a)$ to be $1$ if $a \in Q_1$ and $-1$ otherwise.

\begin{definition}[{\cite[Definition 1.2]{crawley-boeveyshaw}}]\label{def: Lambda q}
Let $Q$ be a quiver, fix a tuple $\mathbf{q} = (q_v) \in (\mathbb{C}^\ast)^\abs{Q_0}$ and a total order on the doubled arrow set $\overline{Q}_1$. The associated \defn{multiplicative preprojective algebra} is
\begin{equation*}
\Lambda^\mathbf{q} \coloneqq \mathbb{C}\overline{Q}[(1 + aa^\ast)\inv]\Big/\inner{\prod_{a \in \overline{Q}_1}(1 + aa^\ast)^{\varepsilon(a)} - \sum_{v \in Q_0}q_ve_v},
\end{equation*}
where the product over $\overline{Q}_1$ is taken with respect to the chosen total order.
\end{definition}

The representation space $\Rep(\Lambda^\mathbf{q},\mathbf{n})$ is the subset of $\Rep(\overline{Q},\mathbf{n})$ where the linear maps at each vertex satisfy the same relations as in $\Lambda^\mathbf{q}$ and the maps representing $1 + aa^\ast$ are invertible.

\begin{definition}\label{def: multiplicative quiver variety}
A \defn{multiplicative quiver variety} is defined as $\mathcal{M}_{\mathbf{q},\mathbf{n}}(Q) = \Rep(\Lambda^\mathbf{q},\mathbf{n}) \GIT \PGL(\mathbf{n})$.
\end{definition}

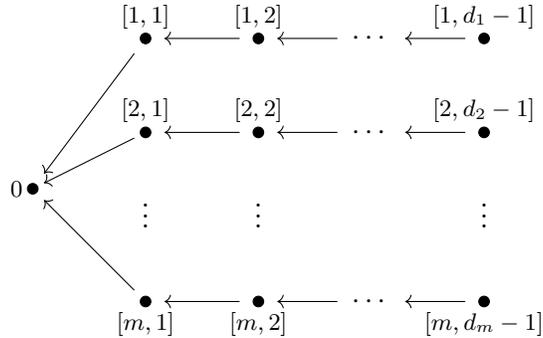
\begin{figure}[H]
\centering
\begin{tikzpicture}
\filldraw (-0.25,0) circle (2pt) node[left]{\footnotesize$0$};
\filldraw (1.25,2) circle (2pt) node[above]{\footnotesize$[1,1]$};
\filldraw (1.25,0.75) circle (2pt) node[above]{\footnotesize$[2,1]$};
\filldraw (1.25,-1.5) circle (2pt) node[below]{\footnotesize$[m,1]$};
\filldraw (2.75,2) circle (2pt) node[above]{\footnotesize$[1,2]$};
\filldraw (2.75,0.75) circle (2pt) node[above]{\footnotesize$[2,2]$};
\filldraw (2.75,-1.5) circle (2pt) node[below]{\footnotesize$[m,2]$};
\filldraw (5.75,2) circle (2pt) node[above]{\footnotesize$[1,d_1 - 1]$};
\filldraw (5.75,0.75) circle (2pt) node[above]{\footnotesize$[2,d_2 - 1]$};
\filldraw (5.75,-1.5) circle (2pt) node[below]{\footnotesize$[m,d_m - 1]$};
\draw[->] (1.1,1.8) -- (-0.1,0.2);
\draw[->] (1.1,0.675) -- (-0.1,0.075);
\draw[->] (1.1,-1.35) -- (-0.1,-0.15);
\draw[->] (2.5,2) -- (1.5,2);
\draw[->] (2.5,0.75) -- (1.5,0.75);
\draw[->] (2.5,-1.5) -- (1.5,-1.5);
\draw[->] (3.8,2) -- (3,2);
\draw[->] (3.8,0.75) -- (3,0.75);
\draw[->] (3.8,-1.5) -- (3,-1.5);
\draw[->] (5.5,2) -- (4.7,2);
\draw[->] (5.5,0.75) -- (4.7,0.75);
\draw[->] (5.5,-1.5) -- (4.7,-1.5);
\node at (1.25,-0.25) {$\vdots$};
\node at (2.75,-0.25) {$\vdots$};
\node at (5.75,-0.25) {$\vdots$};
\node at (4.25,2) {$\cdots$};
\node at (4.25,0.75) {$\cdots$};
\node at (4.25,-1.5) {$\cdots$};
\end{tikzpicture}
\caption{The star-shaped quiver $Q$.}
\label{fig: star-shaped quiver}
\end{figure}

Now let $(A_1, \dots, A_k)$ be a solution of the multiplicative Deligne-Simpson problem \eqref{eq: DS problem} with semi-simple conjugacy classes, where $A_i$ has eigenvalues $\lambda_{ij}$. Associated to this is the multiplicative preprojective algebra of the star-shaped quiver (see Figure \ref{fig: star-shaped quiver} above). Here, $\mathbf{q} = (q_v)$ is given by
\begin{equation}\label{eq: vector q}
q_0 = \prod_{i=1}^k\frac{1}{\lambda_{i1}}, \qquad q_{[i,j]} = \frac{\lambda_{ij}}{\lambda_{i\,j+1}}.
\end{equation}
The corresponding representation of $\Lambda^\mathbf{q}$ is that with vector spaces
\begin{equation*}
V_0 = \mathbb{C}^n, \qquad V_{[i,j]} = \image\big((A_i - \lambda_{i1})\cdots(A_i - \lambda_{ij})\big),
\end{equation*}
meaning the dimension vector is $\mathbf{n} = (n_0, n_{[i,j]})$ where
\begin{equation*}
n_0 = n, \qquad n_{[i,j]} = \rank\big((A_i - \lambda_{i1})\cdots(A_i - \lambda_{ij})\big).
\end{equation*}
Let $a_{ij}$ denote the arrow with tail at the vertex $[i,j]$. The (surjective) linear maps $X_{a_{ij}^\ast}$ are the obvious products, and (injective) linear maps $X_{a_{ij}}$ are inclusions. On the other hand, one can extract a Deligne-Simpson solution from such a representation of $\Lambda^\mathbf{q}$, namely where
\begin{equation*}
A_i = \lambda_{i1}(1 + X_{a_{i1}}X_{a_{i1}^\ast}).
\end{equation*}
A similar construction works more generally for arbitrary conjugacy classes, see \cite{crawley-boevey,crawley-boeveyshaw}.

For generic semi-simple conjugacy classes (in the sense of Definition \ref{def: generic eigendata}), all solutions to the associated Deligne-Simpson problem are irreducible. The proof of \cite[Lemma 8.3]{crawley-boeveyshaw} in that situation implies the following result.

\begin{proposition}\label{prop: character variety isomorphic to quiver variety}
For generic eigenvalues $\lambda_{ij}$, the character variety $\mathfrak{M}_{0,k}$ and the multiplicative quiver variety $\mathcal{M}_{\mathbf{q},\mathbf{n}}(Q)$ are isomorphic. In particular, both are smooth irreducible affine varieties.
\end{proposition}

\begin{remark}\label{rem: our case}
The character variety of interest to us is that of a four-punctured sphere, meaning the corresponding multiplicative preprojective algebra is that of a framed affine Dynkin quiver of type $\widetilde{D}_4$, as sketched in Figure \ref{fig: CM quiver} below. The framed quiver, denoted $Q^\infty$, is obtained by extending one of the legs of the usual affine Dynkin quiver $Q = \widetilde{D}_4$.
\begin{figure}[H]
\centering
\begin{tikzpicture}
\filldraw (0,0) circle (2pt);
\filldraw (0,2) circle (2pt) node[above] {\footnotesize$1$};
\filldraw (-2,0) circle (2pt) node[above] {\footnotesize$2$};
\filldraw (0,-2) circle (2pt) node[below] {\footnotesize$3$};
\filldraw (2,0) circle (2pt) node[above] {\footnotesize$4$};
\filldraw (4,0) circle (2pt) node[above] {\footnotesize$\infty$};
\draw[->] (0,1.8) -- (0,0.2);
\draw[->] (-1.8,0) -- (-0.2,0);
\draw[->] (0,-1.8) -- (0,-0.2);
\draw[->] (1.8,0) -- (0.2,0);
\draw[->] (3.8,0) -- (2.2,0);
\end{tikzpicture}
\caption{The quiver $Q^\infty$ corresponding to the Calogero--Moser space $\mathcal{M}_n$.}
\label{fig: CM quiver}
\end{figure}
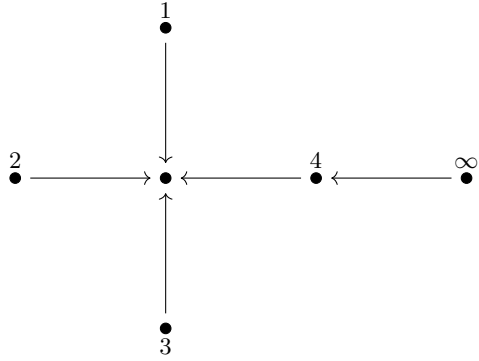
Let $\lambda_{ij}$ be the $j^\text{th}$ eigenvalue of $A_i$ chosen in the \textit{reverse} order compared to \eqref{eq: our eigendata}. With this, the parameters and the dimension vector for $Q^\infty$ become
\begin{equation}\label{eq: our quiver data}
\begin{aligned}
\mathbf{q} &= (q_0,q_1,q_2,q_3,q_4,q_\infty) = (k_0u_0u_nk_n, -k_0^{-2}, -u_0^{-2}, -u_n^{-2}, -k_n^{-2}t^2, t^{-2n}),\\
\mathbf{n} &= (n_0,n_1,n_2,n_3,n_4,n_\infty) = (2n,n,n,n,n,1).
\end{aligned}
\end{equation}
\end{remark}

\begin{corollary}\label{cor: CM space isomorphic to quiver variety}
For generic $\bm{\tau} = (k_0,k_n,t,u_0,u_n)$, the Calogero--Moser space $\mathcal{M}_n$ is isomorphic to the multiplicative quiver variety $\mathcal{M}_{\mathbf{q},\mathbf{n}}(Q^\infty)$ with the quiver data $\mathbf{q}$ and $\mathbf{n}$ as in \eqref{eq: our quiver data}.
\end{corollary}

\section{From Representations of the DAHA to Character Varieties}\label{sec: irreps}

\begin{notation}\label{not: classical DAHA}
Throughout, $\mathcal{H} \coloneqq \mathcal{H}_{1,\bm{\tau}}$ is the DAHA of type $C^\vee C_n$ at the classical level $q = 1$.
\end{notation}

One of the key constructions of \cite{etingofganoblomkov} is a map from the set of irreducible representations of the GDAHAs for $Q = \widetilde{D}_4$, $\widetilde{E}_6, \widetilde{E}_7, \widetilde{E}_8$ and $q = 1$ (that are regular in a certain sense) to a suitable character variety $\mathfrak{M}_{0,k}$. In our situation, for $Q = \widetilde D_4$, this provides us with a map
\begin{equation*}
\Phi : \Irrep'(\mathcal{H}) \to \mathcal{M}_n,
\end{equation*}
where $\Irrep'(\mathcal{H}) \subseteq \Irrep(\mathcal{H})$ is the subset of irreducible representations that restrict to the regular representation of the finite Hecke algebra $H_n \subseteq \mathcal{H}$. Thus, $\dim(V) = 2^nn!$ for all $V \in \Irrep'(\mathcal{H})$.

\begin{remark}\label{rem: Irrep'(H) equal to Irrep(H)}
Although non-obvious a priori, our main result implies (see Corollary \ref{cor: main result} below)
\begin{equation*}
\Irrep'(\mathcal{H}) = \Irrep(\mathcal{H}).
\end{equation*}
\end{remark}

\subsection{The Etingof-Gan-Oblomkov Map}\label{subsec: EGO map}

Recall the one-dimensional representation $\chi : H_n \to \mathbb C$ that assigns $T_i \mapsto t$ ($i = 1,\dots,n-1$) and $T_n \mapsto k_n$. Consider the subalgebra $H_n' \subseteq H_n$ generated by $T_2,\dots,T_n$, and denote by $\chi' : H_n' \to \mathbb{C}$ the restriction of $\chi$ to $H_n'$. Following \cite[\S4.3]{etingofganoblomkov}, for a representation $V \in \Rep(H_n)$, consider
\begin{equation*}
V' = \{f : \chi' \to V\},
\end{equation*}
the space of homomorphisms between $\chi'$ and $V$, viewed as $H_n'$-modules. We may consider $V'$ as a subspace of $V$. Note that if $V \in \Irrep'(\mathcal{H})$, the dimension of the subspace $V'$ is $\dim(V') = 2n$.

\begin{proposition}[{\cite[cf. Proposition 5.2.10]{etingofganoblomkov}}]\label{prop: EGO map}
For $S,S^\dagger$ in \eqref{eq: S and S dagger}, let $Z_i \in \mathcal{H}$ be the elements
\begin{equation}\label{eq: Zi}
Z_1 = T_0, \qquad Z_2 = T_0^\vee, \qquad Z_3 = ST_n^\vee S\inv, \qquad Z_4 = ST_nS^\dagger.
\end{equation}
\begin{enumerate}[label=\rom]
	\item For $V \in \Irrep'(\mathcal{H})$, the elements $Z_i$ commute with $H_n'$ and hence preserve the subspace $V'$.
	\item For $V \in \Irrep'(\mathcal{H})$, denote by $A_i \in \End_{\mathbb{C}}(V')$ the restriction of $Z_i$ onto the subspace $V'$:
	\begin{equation}\label{eq: Ai}
	A_i = Z _i\big\rvert_{V'}, \quad i = 1,2,3,4.
	\end{equation}
	Then, the assignment $V \mapsto (A_1, A_2, A_3, A_4)$ defines a map $\Phi :\ \Irrep'(\mathcal{H})\to \mathcal{M}_n$.
\end{enumerate}
\end{proposition}

\begin{definition}\label{def: EGO map}
The map $\Phi$ from Proposition \ref{prop: EGO map}(ii) is henceforth called the \defn{EGO map}.
\end{definition}

\begin{remark}
Our definition of $\Phi$ is slightly different from the one in \cite{etingofganoblomkov}. In \textit{op. cit}, the authors use the subalgebras $H$ and $H' \subseteq H$ generated respectively by $ST_nS^{-1}, T_1, \dots, T_{n-1}$ and $ST_nS^{-1}, T_1, \dots, T_{n-2}$. These are analogous to the above $H_n$ and $H_n'$ (in fact, their $H$ can also be generated by $T_1, \dots, T_n$, so it is exactly $H_n$). Their $H'$ is used to define a subspace $V' \subseteq V$ in a similar way. Instead of the above elements $Z_i$ \eqref{eq: Zi}, they consider 
\begin{equation*}
\widetilde{U}_i \coloneqq S^\dagger Z_i(S^\dagger)\inv.
\end{equation*}
Their definition of the map $\Phi$ then uses the restriction of these $\widetilde{U}_i$ onto $V'$. But the proof of \cite[Proposition 5.2.10]{etingofganoblomkov} is easily adaptable to our setting. 
\end{remark}

\subsection{From DAHA to Matrices}\label{subsec: coordinate chart}

Our next task is to make the EGO map more explicit. It will be convenient to use a realisation of $\mathcal{H}_{q,\bm{\tau}}$ by operator-valued matrices as in \cite{chalykh}. At the classical level $q = 1$, this provides us with a coordinate chart on the Calogero--Moser space $\mathcal{M}_n$. We begin by considering the vector space
\begin{equation}\label{eq: M}
M \coloneqq \mathbb{C}W \otimes \mathscr{D}_q.
\end{equation}
We identify $M \cong \mathscr{D}_q \rtimes \mathbb{C}W$, writing its elements as
\begin{equation}\label{eq: element of M}
f = \sum_{w \in W}wf_w, \qquad f_w \in \mathscr{D}_q.
\end{equation}
The action of $\mathscr{D}_q \rtimes \mathbb{C}W$ on itself, by left multiplication, provides us with a faithful representation
\begin{equation}\label{eq: matrix representation}
\pi : \mathscr{D}_q \rtimes \mathbb{C}W \to \Mat_{\abs{W} \times \abs{W}}(\mathscr{D}_q).
\end{equation}
Identifying $\mathcal{H}_{q,\bm{\tau}}$ with its image under the Basic Representation, \eqref{eq: matrix representation} implies a representation
\begin{equation*}
\mathcal{H}_{q,\bm{\tau}} \hookrightarrow \Mat_{\abs{W} \times \abs{W}}(\mathscr{D}_q).
\end{equation*}

Let $W'$ be the subgroup of $W$ generated by $\s_i$ and $\s_{ij}$ with $i,j \in \{2,\dots,n\}$. Associated to $W'$ is a subspace of $W'$-invariants, namely
\begin{equation}\label{eq: M' and e'}
M' = e'M, \qquad e' \coloneqq \frac{1}{\abs{W'}}\sum_{w \in W'}w.
\end{equation}
We choose $\{\s_{1i},\s_{1i}^+\}$ for $i=1,\dots,n$ as the coset representatives in $W'\backslash W$, where $\s_{11} \coloneqq \id$ and $\s_{11}^+ \coloneqq \s_1$. Each element $f \in M'$ admits a unique presentation of the form
\begin{equation}\label{eq: element of M'}
f = e'\left(\sum_{i=1}^n\s_{1i}f_i + \sum_{i=1}^n\s_{1i}^+f_i^+\right), \qquad f_i, f_i^+ \in \mathscr{D}_q.
\end{equation}
Any element of $\mathscr{D}_q\rtimes \mathbb C W$ preserving $M'$ can, therefore, be represented by a $\mathscr{D}_q$-valued matrix of size $2n \times 2n$. For example, the action of $X_1$ and $P_1$ on $M'$ are given respectively by the matrices
\begin{equation}\label{eq: matrices of X1 and P1}
X \coloneqq \diag(X_1,\dots,X_n,X_1\inv,\dots,X_n\inv), \qquad P \coloneqq \diag(P_1,\dots,P_n,P_1\inv,\dots,P_n\inv).
\end{equation}
Likewise, the element $\s_1$ as defined in \eqref{eq: sign reversal elements} acts on $M'$ by the matrix
\begin{equation}\label{eq: matrix of s1}
\begin{pmatrix}0 & \Id_n\\\Id_n & 0\end{pmatrix}.
\end{equation}

\begin{lemma}\label{lem: Ai preserve M'}
Let $Z$ denote any of the elements $Z_1$, $Z_2$, $Z_3$, $Z_4$ in \eqref{eq: Zi}. Then, $Z$ preserve $M'$. 
\end{lemma}

\begin{proof}
This is by a standard argument. The Basic Representation implies $T_i - \tau_i = c_i(\bm{X})(s_i - 1)$ for the relevant parameter $\tau_i$, with the rational function $c_i(\bm{X})$ either \eqref{eq: ci(X)} and \eqref{eq: cn(X)}. Notice that each $Z$ commutes with the generators $T_2,\dots,T_n$ of the Hecke subalgebra $H_n'$. Consequently, for $i = 2,\dots,n$, we have
\begin{equation*}
(s_i - 1)Z = \frac{1}{c_i(\bm{X})}(T_i - \tau_i)Z = \frac{1}{c_i(\bm{X})}Z(T_i - \tau_i) = \frac{1}{c_i(\bm{X})}Zc_i(\bm{X})(s_i - 1).
\end{equation*}
Therefore, if $f \in M$ is $W'$-invariant, then so too is $Zf$. In other words, $Z(M') \subseteq M'$.
\end{proof}

Denote by $A_i$ the $\mathscr{D}_q$-valued matrix of size $2n \times 2n$ representing the action of $Z_i$ on $M'$. To state each of them explicitly, we need some notation. Let's begin with two functions of one variable
\begin{equation}\label{eq: a(z) and b(z)}
a(z) = \frac{t\inv - tz}{1 - z}, \qquad b(z) = t - a(z) = \frac{\overline{t}}{1 - z}
\end{equation}
as well as their following two-variable counterparts, where $1 \leq i \neq j \leq n$:
\begin{alignat}{3}
a_{ij} &= a(X_iX_j\inv), \qquad& a_{ij}^+ &= a(X_iX_j), \qquad& a_{ij}^- &= a(X_i\inv X_j\inv),\label{eq: aij}\\
b_{ij} &= b(X_iX_j\inv), \qquad& b_{ij}^+ &= b(X_iX_j), \qquad& b_{ij}^- &= b(X_i\inv X_j\inv).\label{eq: bij}
\end{alignat}

\begin{notation}\label{not: extended indices}
For $i \in \{1,\dots,n\}$, it is convenient to extend such indices to $\{1,\dots,2n\}$ by way of
\begin{equation*}
X_{i+n} \coloneqq X_i\inv, \qquad P_{i+n} \coloneqq P_i\inv.
\end{equation*}
\end{notation}

We also require the following rational functions of one variable:
\begin{align}
u(z) &= k_n^{-1}\frac{(1 - k_nu_nz)(1 + k_nu_n\inv z)}{1 - z^2}, & v(z) &= k_n-u(z)=\frac{\overline{k}_n + \overline{u}_nz}{1 - z^2},\label{eq: u(z) and v(z)}\\
\widetilde{u}(z) &= k_0^{-1}\frac{(1 - q^{1/2}k_0u_0z\inv)(1+q^{1/2}k_0u_0\inv z\inv)}{1 - qz^{-2}}, & \widetilde{v}(z) &= k_0 - \widetilde{u}(z) = \frac{\overline{k}_0 + q^{1/2}\overline{u}_0z\inv}{1 - qz^{-2}}.\label{eq: utilde(z) and vtilde(z)}
\end{align}

Notice that $u$ and $\widetilde{u}$ are respective generalisations of the rational functions $c_i(\bm{X})$ defined in \eqref{eq: cn(X)} and \eqref{eq: c0(X)}. Also, $v$ and $\widetilde{v}$ are related to the Basic Representation, cf. \cite[(4.2.3)]{macdonald}.

We use the notation $u_i \coloneqq u(X_i)$ and $u_i^- \coloneqq (X_i\inv)$, and analogously for the others $v$, $\widetilde{u}$ and $\widetilde{v}$. 

\begin{proposition}\label{prop: matrices of Ai}
Let $A_i$ be the $\mathscr{D}_q$-valued matrix of size $2n \times 2n$ representing the action of $Z_i$ on $M'$. Then, each of them is given by the following respective explicit form:
\begin{equation*}
(A_1)_{ij} = \begin{dcases}
\widetilde{v}_i &\text{if $i = j$}\\
\widetilde{u}_iP_i\inv &\text{if $i - j = \pm n$}\\
0 &\text{otherwise}
\end{dcases},\qquad
(A_2)_{ij} = \begin{dcases}
q^{-1/2}(\widetilde{v}_i-\overline{k}_0)X_i &\text{if $i = j$}\\
q^{-1/2}\widetilde{u}_iP_i\inv X_i\inv &\text{if $i - j = \pm n$}\\
0 &\text{otherwise}
\end{dcases},
\end{equation*}
\begin{equation*}
(A_3)_{ij} =
\begin{dcases}
X_i\inv u_j^-\dprod_{k=1}^{2n}a_{ik}^- &\text{if $i - j = \pm n$}\\
-X_i\inv u_j^-b_{ij}^-a_{ij}\dprod_{k=1}^{2n}a_{jk}^- &\text{if $i - j \neq 0, \pm n$}\\
X_i\inv k_n\inv t^{2-2n} - \sum_{k \neq i}(A_3)_{ik} &\text{if $i = j$}
\end{dcases},
\end{equation*}
\begin{equation*}
(A_4)_{ij} =
\begin{dcases}
u_j^-\dprod_{k=1}^{2n}a_{ik}^-, &\text{if $i - j = \pm n$}\\
u_j^-b_{ij}^+a_{ij}\dprod_{k=1}^{2n}a_{jk}^-, &\text{if $i - j \neq 0, \pm n$}\\
k_n t^{2n-2} - \sum_{k \neq i}(A_4)_{ik}, &\text{if $i = j$}
\end{dcases}.
\end{equation*}
Here, the indices run from $1$ to $2n$, with the convention established in {\normalfont Notation \ref{not: extended indices}}, and the symbol $\dprod\!$ means a product excluding the values of $k$ for which $k - i = 0, \pm n$ and $k - j = 0, \pm n$.
\end{proposition}

\begin{proof}
For $A_1$ and $A_2$, this is straightforward. For example, $T_0 = \widetilde{v}(X_1) + \widetilde{u}(X_1)P_1\inv\s_1$, so its action on $M'$ is easily found from \eqref{eq: matrices of X1 and P1} and \eqref{eq: matrix of s1}. The calculation of $A_4$ can be found in \cite{chalykh}. More precisely, $A_4 = \mathscr{P}\s_1$, where $\mathscr{P}$ is the matrix given in \cite[Proposition 4.3]{chalykh}. For $A_3$, one can use that $Z_3 = X_1\inv Z_4\inv$, so it is sufficient to calculate the action of $Z_4\inv$, which is entirely similar to the way $A_4$ was found in \cite{chalykh}. The result here is 
\begin{equation*}
(A_4\inv)_{ij} =\begin{dcases}
u_j^-\dprod_{k=1}^{2n}a_{ik}^- &\text{if $i - j = \pm n$}\\
-u_j^-b_{ij}^-a_{ij}\dprod_{k=1}^{2n}a_{jk}^- &\text{if $i - j \neq 0, \pm n$}\\
k_n\inv t^{2-2n} - \sum_{k \neq i}(A_4\inv)_{ik} &\text{if $i = j$}
\end{dcases},
\end{equation*}
from which the formula for $A_3$ follows.
\end{proof}

By construction, $A_1$, $A_2$ and $A_3$ satisfy the respective quadratic relations of $T_0$, $T_0^\vee$ and $T_n^\vee$, i.e.
\begin{equation}\label{eq: Hecke relations for A1, A2, A3}
A_1-A_1\inv = \overline{k}_0\Id_{2n}, \qquad A_2-A_2\inv = \overline{u}_0\Id_{2n}, \qquad A_3-A_3\inv = \overline{u}_n\Id_{2n},
\end{equation}
as well as $q^{1/2}A_1A_2A_3A_4 = \Id_{2n}$. For $A_4$, the Hecke relation is deformed as explained below.

First, we have associated to $W'$ the finite Hecke subalgebra $H_n' \subseteq H_n$ generated by $T_2,\dots,T_n$. Let $e_\tau'$ denote the analogue of the Hecke symmetriser $e_\tau$ \eqref{eq: Hecke symmetriser} for this subalgebra, that is 
\begin{equation}\label{eq: Hecke subsymmetriser}
e_\tau' \coloneqq \frac{1}{\sum_{w\in W'}\tau_w^2}\sum_{w \in W'}\tau_wT_w.
\end{equation}
By \cite[(5.5.14), (5.5.15)]{macdonald}, we have $e_\tau = e\mathbf{c}$ and $e_\tau' = e'\mathbf{c}'$ for suitable functions $\mathbf{c},\mathbf{c}'$. Hence, $e'M = e_\tau' M$ and so $e_\tau'$ acts on $M'$ by identity. The action of $e_\tau$ on $M'$ was calculated in \cite{chalykh}.

\begin{lemma}[{\cite[\S3.8 and \S4.4]{chalykh}}]\label{lem: action of Hecke symmetriser on M'}
Consider the column and row vectors $\mathbf{v}, \mathbf{w}$ given by
\begin{equation*}
\mathbf{v} = (1,\dots,1)^T, \qquad \mathbf{w} = (\phi_1, \dots, \phi_{2n}), \qquad 
\phi_i = u_i^-\prod_{k\neq i}^n a_{ki}a_{ki}^-.
\end{equation*}
The Hecke symmetriser $e_\tau$ acts on $M' = e'M$ by the rank-one matrix $\gamma\inv\mathbf{v}\mathbf{w}$, where
\begin{equation}\label{eq: gamma}
\gamma = \frac{t^{2n}-1}{t^2-1}k_n + \frac{1-t^{-2n}}{1-t^{-2}}k_n\inv.
\end{equation}
\end{lemma}

\begin{proof}
Indeed, by the proof of \cite[Proposition 3.3]{chalykh}, the Hecke symmetriser $e_\tau$ acts on $M'$ by $\mathbf{v}\mathbf{w}$ up to a constant factor. Hence, $\mathbf{v}\mathbf{w} = \gamma e_\tau$ on $M'$ for some $\gamma \in \mathbb{C}^\ast$. To calculate this constant, we use the idempotence property $e_\tau^2 = e_\tau$, from which we see that $\gamma = \mathbf{w}\mathbf{v}$. To evaluate this constant, we can take the variables $X_i \to \infty$ within a particular Weyl chamber. This yields
\begin{equation}\label{eq: trace of vw}
\mathbf{w}\mathbf{v} = \frac{t^{2n}-1}{t^2-1}k_n + \frac{1-t^{-2n}}{1-t^{-2}}k_n\inv,
\end{equation}
as stated.
\end{proof}

\begin{lemma}\label{lem: Z4 relation}
In $\mathcal{H}_{q,\bm{\tau}}$, the element $Z_4 = ST_nS^\dagger$ satisfies the relation
\begin{equation}\label{eq: Z4 relation}
(tZ_4 - t\inv Z_4\inv - k_nt\inv + k_n\inv t)e_\tau' = (k_nt^{2n-1} - k_n\inv t^{1-2n} + k_n\inv t - k_nt\inv)e_\tau.
\end{equation}
\end{lemma}

\begin{proof}
Act by both sides on $M' = e'M = e_\tau'M$. The left-hand side of \eqref{eq: Z4 relation} is easily computed on $M'$ since we know $A_4$ and $A_4\inv$ explicitly. It is then straightforward to check that 
\begin{equation}\label{eq: A4 relation}
tA_4 - t\inv A_4\inv - (k_nt\inv - k_n\inv t)\Id_{2n} = (t - t\inv)\mathbf{v}\mathbf{w}.
\end{equation}
Since $\mathbf{v}\mathbf{w} = \gamma e_\tau$ on $M'$, we conclude that \eqref{eq: Z4 relation} is valid on $M'$ also. Because $\mathcal{H}_{q,\bm{\tau}}$ acts faithfully on $M$, the relation \eqref{eq: Z4 relation} then follows. 
\end{proof}

\begin{corollary}\label{cor: chart on CM space}
For $q = 1$, the quadruple $(A_1,A_2,A_3,A_4)$ from {\normalfont Proposition \ref{prop: matrices of Ai}} represents a point on the Calogero--Moser space. This gives a coordinate chart on $\mathcal{M}_n$ with $2n$ coordinates $X_i, P_i$.
\end{corollary}

\begin{proof}
Set $X = A_1A_2$, $Y = A_4A_1$ and $T = A_2$; so \eqref{eq: Hecke relations for A1, A2, A3}, \eqref{eq: trace of vw} and \eqref{eq: A4 relation} imply \eqref{eq: T relation}--\eqref{eq: wv relation}.
\end{proof}

One can also see that the action by $w \in W$ on $(\bm{P},\bm{X})$ is equivalent to conjugating $(A_1,A_2,A_3,A_4)$ by the permutation matrix representing $w$. Therefore, we have a well-defined map
\begin{equation}\label{eq: Upsilon}
\Upsilon : \mathcal{U}/W \to \mathcal{M}_n, \qquad (\bm{P},\bm{X}) \mapsto (A_1,A_2,A_3,A_4),
\end{equation}
where
\begin{equation}\label{eq: U}
\mathcal{U} \coloneqq (\mathbb{C}^\ast)^n \times ((\mathbb{C}^\ast)^n \setminus D), \qquad \text{with}\ D = \{\bm{X} \in (\mathbb{C}^\ast)^n : \delta(\bm{X}) = 0\}.
\end{equation}

\subsection{The EGO Map on a Chart}\label{subsec: local coordinates}

Let us explain how the above calculation is related to the EGO map. At the classical level $q = 1$, recall the Basic Representation gives an injective map $\beta : \mathcal{H} \to \mathbb{C}[\bm{X}^{\pm1},\bm{P}^{\pm1}]_{\delta(\bm{X})} \rtimes \mathbb{C}W$. Thus, one can construct a family of irreducible representations of $\mathcal{H}$ from irreducible representations of $\mathbb{C}[\bm{X}^{\pm1},\bm{P}^{\pm1}]_{\delta(\bm{X})} \rtimes \mathbb{C}W$ by restriction, cf. \cite[\S3.3]{oblomkovJun04}. To alleviate some of the notation, let $R \coloneqq \mathbb{C}[\bm{X}^{\pm1},\bm{P}^{\pm1}]_{\delta(\bm{X})}$ and $\widehat{R} \coloneqq R \rtimes \mathbb{C}W$. Now, pick a point $(\bm{\mu},\bm{\nu}) \in \mathcal{U}$ and define a one-dimensional representation $\chi_{\bm{\mu},\bm{\nu}}$ of $R$ by $\chi_{\bm{\mu},\bm{\nu}} : f(\bm{X},\bm{P})\mapsto f(\bm{\nu},\bm{\mu})$. From this, we can induce a finite-dimensional module:
\begin{equation}\label{eq: induced module V}
V_{\bm{\mu},\bm{\nu}} \coloneqq \Ind_R^{\widehat{R}}\chi_{\bm{\mu},\bm{\nu}} = \widehat{R} \otimes_{R} \chi_{\bm{\mu},\bm{\nu}}.
\end{equation}

\begin{proposition}\label{prop: induced module V is in Irrep'(H)}
Assume $\delta_{\bm{\tau}}(\bm{\nu}) \neq 0$. Then, viewed as an $\mathcal{H}$-module, $V_{\bm{\mu},\bm{\nu}}$ belongs to $\Irrep'(\mathcal{H})$ and its image under the EGO map is represented by the tuple $(A_1,A_2,A_3,A_4)$, where the $A_i$ are the matrices from {\normalfont Proposition \ref{prop: matrices of Ai}} with $q = 1$ under the specialisation $P_i = \mu_i$ and $X_i = \nu_i$. 
\end{proposition}

\begin{proof}
Since $\delta(\bm{\nu})\delta_{\bm{\tau}}(\bm{\nu}) \neq 0$, one can use Corollary \ref{cor: localisation of the DAHA by delta} to view $V_{\bm{\mu},\bm{\nu}}$ as a module over
\begin{equation*}
\mathcal{H}_{\delta(\bm{X})\delta_{\bm{\tau}}(\bm{X})} \cong \mathbb{C}[\bm{X}^{\pm1},\bm{P}^{\pm1}]_{\delta(\bm{X})\delta_{\bm{\tau}}(\bm{X})} \rtimes \mathbb{C}W.
\end{equation*}
It is clearly irreducible as a $(\mathbb{C}[\bm{X}^{\pm1},\bm{P}^{\pm1}]_{\delta(\bm{X})\delta_{\bm{\tau}}(\bm{X})} \rtimes \mathbb{C}W)$-module, and thus also as an $\mathcal{H}$-module (cf. \cite[Proposition 3.3]{oblomkovJun04}). It is isomorphic to $\mathbb{C}W$ as a $W$-module, by construction. Hence, $V_{\bm{\mu},\bm{\nu}}$ is isomorphic to the regular representation as an $H_n$-module by a deformation argument (our genericity assumption from Definition \ref{def: generic paramteters} on the parameters implies that $H_n$ is semi-simple). This establishes that $V_{\bm{\mu},\bm{\nu}} \in \Irrep'(\mathcal{H})$.

To apply the EGO map to the module $V_{\bm{\mu},\bm{\nu}}$, one needs to consider the subspace $V'$ on which the Hecke subalgebra $H_n'$ acts according to the character $\chi'$, see \S\ref{subsec: EGO map}. This means that $(T_i - \tau_i)v = 0$ for $v \in V'$ and $i = 2,\dots,n$. We claim that $V' = e'V_{\mu,\nu}$, where $e'$ is the symmetriser \eqref{eq: M' and e'}. Indeed, $T_i-\tau_i=c_i(\mathbf{X})(s_i-1)$, where $c_i(\bm{X})$ is invertible on $V_{\bm{\mu},\bm{\nu}}$ due to the condition $\delta(\bm{\nu})\delta_{\bm{\tau}}(\bm{\nu}) \neq 0$. Hence, $(s_i - 1)v = 0$ for $v \in V'$ and $i = 2,\dots,n$. This establishes that $V' = e'V_{\bm{\mu},\bm{\nu}}$.

By definition, a basis of $V_{\bm{\mu},\bm{\nu}}$ is given by elements $w \otimes 1$ where $w \in W$ and with $1$ denoting a basis vector in $\chi_{\bm{\mu},\bm{\nu}}$. Elements of $V'$ can then be written similarly to elements \eqref{eq: element of M'} of $M'$, i.e.
\begin{equation}\label{eq: element of V'}
v = e'\left(\sum_{i=1}^n\s_{1i} \otimes f_i + \sum_{i=1}^n\s_{1i}^+ \otimes f_i^+\right), \qquad f_i, f_i^+ \in \mathbb{C}.
\end{equation}
Any element of $\mathbb{C}[\bm{X}^{\pm1},\bm{P}^{\pm1}]_{\delta(\bm{X})\delta_{\bm{\tau}}(\bm{X})} \rtimes \mathbb{C}W$ preserving $V'$ can be represented by a matrix of size $2n \times 2n$. For example, the action of $X_1$ and $P_1$ on $V'$ are given by the respective matrices
\begin{equation*}\label{eq: matrices of X1 and P1 on V'}
X \coloneqq \diag(\nu_1,\dots,\nu_n,\nu_1\inv,\dots,\nu_n\inv), \qquad P \coloneqq \diag(\mu_1,\dots,\mu_n,\mu_1\inv,\dots,\mu_n\inv).
\end{equation*}
Clearly, we are in the same setting as in \S\ref{subsec: coordinate chart}, with the only difference being that $X_i$ and $P_i$ are specialised to $\nu_i$ and $\mu_i$, respectively. Hence, the action of $Z_1$, $Z_2$, $Z_3$, $Z_4$ on $V'$ is found by specialising the formulae in Proposition \ref{prop: matrices of Ai}.
\end{proof}

\begin{remark}\label{rem: V isomorphic to induced module}
The $\mathcal{H}$ module $V_{\bm{\mu},\bm{\nu}}$ admits the following interpretation, cf. \cite[Lemma 6.1]{oblomkovJun04}. Assuming that $\delta(\bm{\nu})\delta_{\bm{\tau}}(\bm{\nu}) \neq 0$, one can consider a one-dimensional representation $\chi_{\bm{\mu},\bm{\nu}}$ of the (commutative) $W$-invariant algebra $\mathbb{C}[\bm{X}^{\pm1},\bm{P}^{\pm1}]_{\delta(\bm{X})\delta_{\bm{\tau}}(\bm{X})}^We_\tau$ defined by $f(\bm{X},\bm{P})e_\tau \mapsto f(\bm{\mu},\bm{\nu})$. We can further restrict $\chi_{\bm{\mu},\bm{\nu}}$ to the spherical subalgebra $e_\tau\mathcal{H}e_\tau$ using Corollary \ref{cor: localisation of the DAHA by delta}. Hence,
\begin{equation}\label{eq: V isomorphic to induced module}
V_{\bm{\mu},\bm{\nu}} \cong \mathcal{H}e_\tau \otimes_{e_\tau\mathcal{H}e_\tau}\chi_{\bm{\mu},\bm{\nu}}.
\end{equation}
\end{remark}

\section{Coordinates on the Character Variety}\label{sec: coordinates}

The map $\Upsilon$ \eqref{eq: Upsilon} defines coordinates on the Calogero--Moser space $\mathcal{M}_n$. More precisely, let
\begin{equation}\label{eq: U tau}
\mathcal{U}_{\bm{\tau}} \coloneqq (\mathbb{C}^\ast)^n\times ((\mathbb{C}^\ast)^n\setminus (D \cup D_{\bm{\tau}})), \qquad\text{with}\ D_{\bm{\tau}} = \{\bm{X} \in (\mathbb{C}^\ast)^n : \delta_{\bm{\tau}}(\bm{X}) = 0\}.
\end{equation}

\begin{proposition}\label{prop: Upsilon is injective on subset}
The map $\Upsilon$ is injective on the subset $\mathcal{U}_{\bm{\tau}}/W$.
\end{proposition}

\begin{proof}
Let $\Upsilon(\bm{P},\bm{X}) = (A_1,A_2,A_3,A_4)$ and $\Upsilon(\bm{P}',\bm{X}') = (A_1',A_2',A_3',A_4')$ be such that
\begin{equation}\label{eq: conjugate CM point}
\Upsilon(\bm{P},\bm{X}) = \Upsilon(\bm{P}',\bm{X}') \qquad\text{and}\qquad (A_1,A_2,A_3,A_4) = g(A_1',A_2',A_3',A_4')g\inv ,
\end{equation}
for some $g \in \GL_{2n}(\mathbb{C})$. From the definitions of $A_1$ and $A_2$, the matrix $X = A_1A_2$ represents the action of $X_1$ on $M'$, so it has the diagonal form \eqref{eq: matrices of X1 and P1}. Therefore, up to the action by $W$, we may assume $\bm{X} = \bm{X}'$, which then forces $g$ to be diagonal and implies both $A_3 = A_3'$ and $A_4 = A_4'$. Combining this with the last entries of the tuples in \eqref{eq: conjugate CM point}, we have $gA_4g\inv = A_4$ (with $g$ diagonal). However, the condition $\delta_{\bm{\tau}}(\bm{X}) \neq 0$ implies that all the off-diagonal entries of $A_4$ are non-zero. This forces $g$ to be a multiple of identity. Consequently, $A_1 = A_1'$, from which $\bm{P} = \bm{P}'$ (this uses the fact that $\widetilde{u}_i = \widetilde{u}(X_i) \neq 0$ on $\mathcal{U}_0$).
\end{proof}

By Corollary \ref{cor: chart on CM space}, having an injective algebraic map $\Upsilon : \mathcal{U}_{\bm{\tau}}/W \to \mathcal{M}_n$, we would like to characterise its image. Both sides have dimension $2n$, so irreducibility of $\mathcal{M}_n$ tells us that the image of $\Upsilon$ is automatically dense in $\mathcal{M}_n$. Recall that at every point of $\mathcal{U}_0$, $X = A_1A_2$ is a diagonalisable matrix of the form \eqref{eq: matrices of X1 and P1}. In particular, its eigenvalues are paired-off into reciprocals. This fact is actually true globally.

\begin{lemma}\label{lem: eigenvalues of X are pairewise reciprocal}
For any $(A_1,A_2,A_3,A_4) \in \mathcal{M}_n$, $X = A_1A_2$ satisfies $\tr X^k =\tr X^{-k}$ for all $k \in \mathbb{Z}$, and its eigenvalues appear in pairs $(x_i,x_i\inv)$ for each $i=1,\dots,n$. Hence, there is a global map
\begin{equation*}
\mathcal{M}_n \to (\mathbb{C}^\ast)^n/W, \qquad (A_1,A_2,A_3,A_4) \mapsto (x_1,\dots,x_n).
\end{equation*}
\end{lemma}

\begin{proof}
By the above, the property $\tr X^k = \tr X^{-k}$ holds on a dense subset of $\mathcal{M}_n$ (on the image of $\Upsilon$); it is therefore true everywhere. Also, generically on $\mathcal{M}_n$, the eigenvalues of $X$ are pairwise reciprocal. This obviously remains true even when some of the $x_i$ degenerate to $\pm1$.
\end{proof}

In other words, the eigenvalues of $X = A_1A_2$ are always of the form $(x_1,\dots,x_n, x_{1}\inv,\dots,x_n\inv)$, which defines a point $\mathbf{x} \coloneqq (x_1,\dots,x_n) \in (\mathbb{C}^\ast)^n$, but only up to the action of $W$. Our next goal is to establish the following result.

\begin{proposition}\label{prop: Upsilon is isomorphic on a subset}
The map $\Upsilon$ defines an isomorphism $\mathcal{U}_{\bm{\tau}}/W \cong \mathcal{M}_n^{\bm{\tau}}$, where
\begin{equation}\label{eq: Mn tau}
\mathcal{M}_n^{\bm{\tau}} \coloneqq \{(A_1,A_2,A_3,A_4)\in \mathcal{M}_n : \delta(\mathbf{x})\delta_{\bm{\tau}}(\mathbf{x}) \neq 0\}.
\end{equation}
\end{proposition}

In view of Proposition \ref{prop: Upsilon is injective on subset}, we need only to show $\Upsilon$ is onto. Given a point $(A_1,A_2,A_3,A_4) \in \mathcal{M}_n^{\bm{\tau}}$, we show that it is representation by the matrices from Proposition \ref{prop: matrices of Ai} in a suitable basis (at the classical level $q = 1$). Unlike in \cite{oblomkovJun04}, we cannot easily solve the matrix equations determining the $A_i$, so we instead use the interpretation involving the Deligne-Simpson problem and the representation of a multiplicative preprojective algebra (see \S\ref{subsec: quiver varieties}).

\subsection{Two Deligne-Simpson Problems}\label{subsec: two DS problems}

The condition $\delta(\mathbf{x}) \neq 0$ guarantees that $X = A_1A_2$ is diagonalisable, so we readily assume that
\begin{equation}\label{eq: diagonal X}
X = \diag(x_1,\dots,x_n,x_1\inv,\dots,x_n\inv), \qquad\text{with}\ \delta(\mathbf{x})\delta_{\bm{\tau}}(\mathbf{x}) \neq 0.
\end{equation}

The problem of determining $A_1$, $A_2$, $A_3$, $A_4$ then splits into two separate problems, by considering a so-called \textit{pair of pants} decomposition of the four-punctured sphere.

\begin{problem}\label{prob: DS problem for X}
Let $X$ be the diagonal matrix \eqref{eq: diagonal X} and $C_i = [\Lambda_i] \subseteq \GL_{2n}(\mathbb{C})$ the conjugacy classes chosen in accordance with \eqref{eq: our eigendata}.
\begin{enumerate}[label=\rom]
	\item Find $A_1 \in C_1$ and $A_2 \in C_2$ with $A_1A_2X\inv = \Id_{2n}$.
	\item Find $A_3 \in C_3$ and $A_4 \in C_4$ with $XA_3A_4 = \Id_{2n}$.
\end{enumerate}
\end{problem}

Each of Problems \ref{prob: DS problem for X}(i) and (ii) is a Deligne-Simpson problem on a three-punctured sphere. The corresponding star-shaped quivers are shown below in Figure \ref{fig: DS problem for X}.

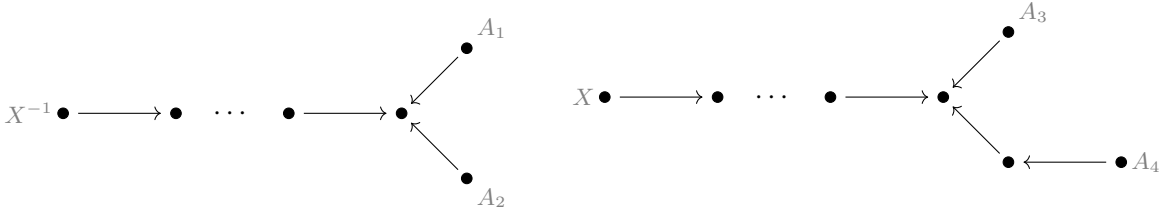
\begin{figure}[H]
\captionsetup[subfigure]{justification=centering}
\begin{subfigure}[b]{0.45\linewidth}
\centering
\begin{tikzpicture}
\filldraw (0,0) circle (2pt);
\filldraw (0.866,0.866) circle (2pt) node[above right,gray] {\footnotesize $A_1$};
\filldraw (0.866,-0.866) circle (2pt) node[below right,gray] {\footnotesize $A_2$};
\filldraw (-1.5,0) circle (2pt);
\filldraw (-3,0) circle (2pt);
\filldraw (-4.5,0) circle (2pt) node[left,gray] {\footnotesize $X\inv$};
\node at (-2.25,0) {$\cdots$};
\draw[->] (0.75,0.75) -- (0.125,0.125);
\draw[->] (0.75,-0.75) -- (0.125,-0.125);
\draw[->] (-1.3,0) -- (-0.2,0);
\draw[->] (-4.3,0) -- (-3.2,0);
\end{tikzpicture}
\vspace{2mm}
\caption{The quiver $Q_1$ for Problem \ref{prob: DS problem for X}(i).}
\label{subfig: quiver for DS problem 1 for X}
\end{subfigure}\hfill
\begin{subfigure}[b]{0.5\linewidth}
\centering
\begin{tikzpicture}
\filldraw (0,0) circle (2pt);
\filldraw (0.866,0.866) circle (2pt) node[above right,gray] {\footnotesize $A_3$};
\filldraw (0.866,-0.866) circle (2pt);
\filldraw (2.366,-0.866) circle (2pt) node[right,gray] {\footnotesize $A_4$};
\filldraw (-1.5,0) circle (2pt);
\filldraw (-3,0) circle (2pt);
\filldraw (-4.5,0) circle (2pt) node[left,gray] {\footnotesize $X$};
\node at (-2.25,0) {$\cdots$};
\draw[->] (0.75,0.75) -- (0.125,0.125);
\draw[->] (0.75,-0.75) -- (0.125,-0.125);
\draw[->] (2.166,-0.866) -- (1.066,-0.866);
\draw[->] (-1.3,0) -- (-0.2,0);
\draw[->] (-4.3,0) -- (-3.2,0);
\end{tikzpicture}
\vspace{2mm}
\caption{The quiver $Q_2$ for Problem \ref{prob: DS problem for X}(ii).}
\label{subfig: quiver for DS problem 2 for X}
\end{subfigure}\hfill
\vspace{5mm}
\caption{The quivers associated with Problem \ref{prob: DS problem for X}.}
\label{fig: DS problem for X}
\end{figure}

The dimension vectors of the quivers $Q_1$ and $Q_2$, reading vertices left-to-right, are
\begin{equation}\label{eq: dimension vectors of Q1 and Q2}
\mathbf{n}_1 = (1,\dots,2n-1,2n,n,n) \qquad\text{and}\qquad \mathbf{n}_2 = (1,\dots,2n-1,2n,n,n,1).
\end{equation}

We know already that each of these problems has a solution. As we shall explain, the solution is unique up to conjugation, and the corresponding multiplicative quiver variety in each case consists of a point. Additional $n$ coordinates $p_1,\dots,p_n$ arise corresponding to the different ways of \say{glueing} the two solutions together. The following result will be used.

\begin{theorem}[{\cite[Lemma 1.5 and Theorem 1.8]{crawley-boeveyshaw}}]\label{thrm: preprojective algbera modules}
The dimension of a simple $\Lambda^\mathbf{q}$-module is a positive root for the corresponding quiver $Q$. Moreover, if $\alpha$ is the dimension of a $\Lambda^\mathbf{q}$-module, then $\alpha$ admits a decomposition into a sum of positive roots $\alpha = \beta + \gamma + \cdots$ with $\mathbf{q}^\beta = \mathbf{q}^\gamma = \cdots = 1$.
\end{theorem}

\subsection{Problem \ref*{prob: DS problem for X}(i)}\label{subsec: first DS problem}

The quiver $Q = Q_1$ is a finite Dynkin quiver of type $D_{2n+2}$. We consider the multiplicative preprojective algebra $\Lambda^\mathbf{q}$ with the parameters $\mathbf{q}$ determined from the eigenvalues of $X$, $A_1$ and $A_2$ (see \eqref{eq: vector q} above). We use the eigenvalues of $X$ in the following order (again going from left-to-right along $Q$): $x_n\inv,\dots,x_1\inv,x_n,\dots,x_1$. We are interested in $\Lambda^\mathbf{q}$-modules of dimension
\begin{equation*}
\alpha = \mathbf{n}_1 = (1,\dots,2n-1,2n,n,n).
\end{equation*}

By Theorem \ref{thrm: preprojective algbera modules}, $\alpha$ must be a sum of positive roots of $Q$. Since the arrows of $Q$ are represented by injective/surjective maps, the support of each summand $\beta,\gamma,\dots$ should include the central node. This gives us the following possibilities for the summands in Figure \ref{fig: positive roots with central support} below.

\begin{figure}[H]
\centering
\begin{tikzpicture}
\node at (0.866,0.866) {\footnotesize$0$};
\draw[->] (0.75,0.75) -- (0.125,0.125);
\node at (0.866,-0.866) {\footnotesize$1$};
\draw[->] (0.75,-0.75) -- (0.125,-0.125);
\node at (0,0) {\footnotesize$1$};
\draw[->] (-1.3,0) -- (-0.2,0);
\node at (-1.5,0) {\footnotesize$1$};
\node at (-2.25,0) {$\cdots$};
\node at (-3,0) {\footnotesize$1$};
\draw[->] (-4.3,0) -- (-3.2,0);
\node at (-4.5,0) {\footnotesize$0$};
\node at (-5.25,0) {$\cdots$};
\node at (-6,0) {\footnotesize$0$};
\begin{scope}[xshift=9cm]
\node at (0.866,0.866) {\footnotesize$1$};
\draw[->] (0.75,0.75) -- (0.125,0.125);
\node at (0.866,-0.866) {\footnotesize$1$};
\draw[->] (0.75,-0.75) -- (0.125,-0.125);
\node at (0,0) {\footnotesize$1$};
\draw[->] (-1.3,0) -- (-0.2,0);
\node at (-1.5,0) {\footnotesize$1$};
\node at (-2.25,0) {$\cdots$};
\node at (-3,0) {\footnotesize$1$};
\draw[->] (-4.3,0) -- (-3.2,0);
\node at (-4.5,0) {\footnotesize$0$};
\node at (-5.25,0) {$\cdots$};
\node at (-6,0) {\footnotesize$0$};
\end{scope}
\begin{scope}[xshift=5.5cm,yshift=-3cm]
\node at (0.866,0.866) {\footnotesize$1$};
\draw[->] (0.75,0.75) -- (0.125,0.125);
\node at (0.866,-0.866) {\footnotesize$1$};
\draw[->] (0.75,-0.75) -- (0.125,-0.125);
\node at (0,0) {\footnotesize$2$};
\node at (-0.75,0) {$\cdots$};
\node at (-1.5,0) {\footnotesize$2$};
\draw[->] (-2.8,0) -- (-1.7,0);
\node at (-3,0) {\footnotesize$1$};
\node at (-3.75,0) {$\cdots$};
\node at (-4.5,0) {\footnotesize$1$};
\draw[->] (-5.8,0) -- (-4.7,0);
\node at (-6,0) {\footnotesize$0$};
\node at (-6.75,0) {$\cdots$};
\node at (-7.5,0) {\footnotesize$0$};
\end{scope}
\end{tikzpicture}
\vspace{2mm}
\caption{Positive roots supported at the central node.}
\label{fig: positive roots with central support}
\end{figure}
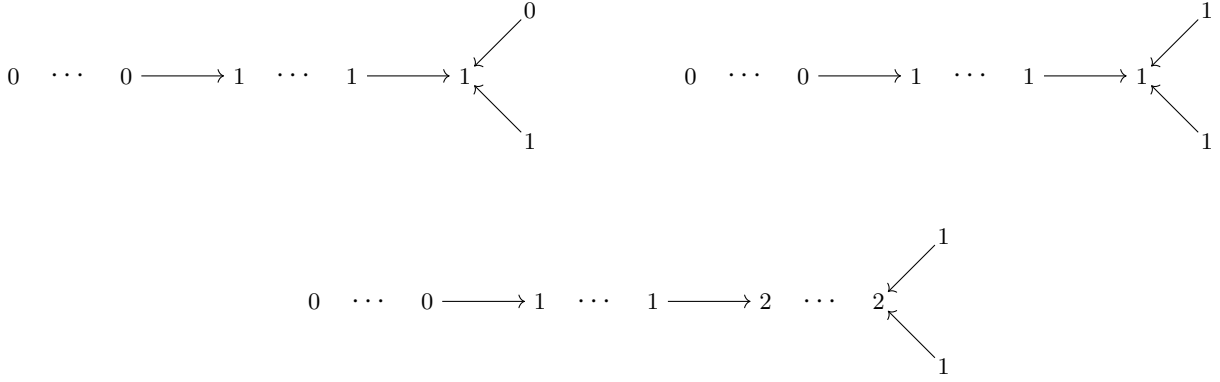

The first two types of summand can be ruled out because $\mathbf{q}^\beta \neq 1$ in these cases. For example, in the first case, we have a one-dimensional subspace at the central node which has to be an eigenspace for $X\inv$, $A_1$ and $A_2$ with respective eigenvalues $x_i^{\pm1}$, $k_0$ and $u_0$. Since $A_1A_2X\inv = \Id_{2n}$, it follows that $x_i^{\pm1}k_0u_0 = 1$, a contradiction to the assumption $\delta_{\bm{\tau}}(\mathbf{x}) \neq 0$ made in \eqref{eq: diagonal X}.

Consequently, each positive root summand $\beta,\gamma,\dots$ should be of the final type. In this case, we have a two-dimensional subspace at the central node. As above, this corresponds to $X\inv$ having eigenvalues $x_i^{\pm1}$, $x_j^{\pm1}$ for some $i,j$. The condition $\mathbf{q}^\beta = 1$ then forces these eigenvalues to be reciprocal to one another. Hence, $\beta$ must be one of $e_{1\,n+1},e_{2\,n+2},\dots,e_{n\,2n}$, where $e_{ij}$ is the $2n$-tuple associated with the last positive root in Figure \ref{fig: positive roots with central support}, that is
\begin{equation*}
e_{ij} = (0,\dots,0,\smallunderbrace{1}_{i^\text{th}},\dots,1,\smallunderbrace{2}_{j^\text{th}},\dots,2,1,1).
\end{equation*}
Hence, the only allowed decomposition of $\alpha = \mathbf{n}_1$ into a sum of simple roots is
\begin{equation}\label{eq: decomposition of n1}
\alpha = e_{1\,n+1} + e_{2\,n+2} + \cdots + e_{n\,2n}.
\end{equation}
As the next lemma shows, this determines the $\Lambda^\mathbf{q}$-module uniquely.

\begin{lemma}\phantomsection\label{lem: preprojective algbera determined uniquely}
\begin{enumerate}[label=\rom]
	\item For each root $\beta_i \coloneqq e_{i\,n+i}$, there is a unique $\Lambda^\mathbf{q}$ module $M_i$ of dimension $\beta_i$.
	\item If $M$ is a $\Lambda^\mathbf{q}$-module of dimension $\alpha = \mathbf{n}_1$, then $M$ is isomorphic to the direct sum of $M_i$ for $i=1,\dots,n$. Hence, up to isomorphism, there is one $\Lambda^\mathbf{q}$-module of dimension $\alpha = \mathbf{n}_1$.
\end{enumerate}
\end{lemma}

\begin{proof}
(i) Existence and uniqueness of $M_i$ follow from \cite[Theorem 1.9]{crawley-boeveyshaw}, using the fact that $Q$ is of finite Dynkin type and thus its roots are real. Alternatively, the problem of constructing $M_i$ can be interpreted as the description of the monodromy of the classical hypergeometric equation, where uniqueness (so-called \textit{rigidity} in \textit{op. cit.}) as well as the existence are well know.

(ii) From \eqref{eq: decomposition of n1}, it follows that the composition series of $M$ consists of the modules $M_1,\dots,M_n$. By using \cite[Theorem 1.6]{crawley-boeveyshaw}, one can check that $\Ext_{\Lambda^\mathbf{q}}^1(M_i,M_j) = 0$ for $i \neq j$. Hence, $M$ splits into the direct sum of the $M_i$.
\end{proof}

\subsection{Problem \ref*{prob: DS problem for X}(ii)}\label{subsec: second DS problem}

The quiver $Q = Q_2$ is not of finite Dynkin type, but we can use the fact it becomes Dynkin after the removal of the extending vertex (the framing). We again begin by considering the multiplicative preprojective algebra $\Lambda^\mathbf{q}$ with the parameters $\mathbf{q}$ determined from the eigenvalues of $X$, $A_3$ and $A_4$. We are interested in $\Lambda^\mathbf{q}$-modules of dimension
\begin{equation*}
\alpha = \mathbf{n}_2 = (1,\dots,2n-1,2n,n,n,1).
\end{equation*}

\begin{lemma}\label{lem: preprojective algebra module is simple}
Any $\Lambda^\mathbf{q}$-module of dimension $\alpha = \mathbf{n}_2$ is simple.
\end{lemma}

\begin{proof}
If it were not simple, $\alpha$ would be a sum $\beta + \gamma + \cdots$ of roots of $Q$ with $\mathbf{q}^\beta = \mathbf{q}^\gamma = \cdots = 1$. At least one of them, $\beta$ say, must have a zero entry corresponding to the extended vertex. Hence, it is a root of the finite Dynkin graph of type $D_{2n+2}$. The same analysis as for Problem \ref{prob: DS problem for X}(i) can then be carried out; the difference is that the eigenvalues for $A_4$ have powers of $t$, preventing $\mathbf{q}^\beta = 1$, a contradiction.
\end{proof}

By \cite[Theorem 1.8]{crawley-boeveyshaw}, $\alpha = \mathbf{n}_2$ is a root of $Q_2$. The dimension of the multiplicative quiver variety can be found using \cite[Theorem 1.10]{crawley-boeveyshaw}. Namely, it is $2p(\alpha)$, where $p(\alpha) \coloneqq 1 - q(\alpha)$ and $q(\alpha)$ being the Tits form. By direct calculation, $q(\alpha) = 1$ and hence $p(\alpha) = 0$. Therefore, the variety $\mathcal{M}_{\mathbf{q},\alpha}(Q_2)$ is zero-dimensional. This is the so-called \textit{rigid} case of the Deligne-Simpson problem, in which case its solution is unique; see \cite[Theorem 1.5]{crawley-boevey}.

\subsection{Proof of Proposition \ref*{prop: Upsilon is isomorphic on a subset}}\label{subsec: proof of surjectivity}

Recall we must show that every point in $\mathcal{M}_n^{\bm{\tau}}$ \eqref{eq: Mn tau} is represented by $(A_1,A_2,A_3,A_4)$ in accordance with the formulae for the matrices in Proposition \ref{prop: matrices of Ai}, at the classical level $q = 1$.

\begin{proof}[Proof of Proposition \ref*{prop: Upsilon is isomorphic on a subset}]
First, we set all $P_i = 1$ and denote the corresponding matrices $A_i^\bullet$. From the construction of these matrices, we know that $A_1^\bullet A_2^\bullet = X$ and $A_3^\bullet A_4^\bullet = X\inv$, where $X$ is the matrix \eqref{eq: matrices of X1 and P1}. This gives us solutions to Problems \ref{prob: DS problem for X}(i) and (ii). Assuming $\delta_{\bm{\tau}}(\bm{X}) \neq 0$, each of these Deligne-Simpson solutions is unique up to conjugation that leaves $X$ unchanged (i.e. up to conjugation by diagonal matrices). The general solution to these problems, given $X$, is
\begin{equation*}
(A_1,A_2,A_3,A_4) = (CA_1^\bullet C\inv, CA_2^\bullet C\inv, DA_3^\bullet D\inv, DA_4^\bullet D\inv),
\end{equation*}
with $C$ and $D$ diagonal. Using simultaneous conjugation, we can set $D = \Id$, meaning $A_3$ and $A_4$ are already in the desired form. On the other hand, matrices $A_1^\bullet$ and $A_2^\bullet$ have most of their entries equal to zero, so that they commute with any matrix of the form $C = \diag(c_1,\dots,c_n,c_1,\dots,c_n)$. Hence, without loss of generality, we may restrict the choice of $C$ to $\diag(1,\dots,1,P_1,\dots,P_n)$. It is easy to check that conjugating $A_1^\bullet$ and $A_2^\bullet$ by such $C$ results in precisely $A_1$ and $A_2$ from Proposition \ref{prop: matrices of Ai}. Therefore, any point in $\mathcal{M}_n^{\bm{\tau}}$ with given $X$ satisfying \eqref{eq: diagonal X} is represented by the $A_i$ we calculated in \S\ref{subsec: coordinate chart}.
\end{proof}

\subsection{Duality and the Second Chart}\label{subsec: duality and the second chart}

Recall the duality isomorphism $\varepsilon$ from Theorem \ref{thrm: duality isomorphism}. It acts on the elements \eqref{eq: Zi} by
\begin{equation}\label{eq: dual of Zi}
\varepsilon(Z_1) = (Z_3')\inv, \qquad \varepsilon(Z_2) = (Z_2')\inv, \qquad \varepsilon(Z_3) = (Z_1')\inv, \qquad \varepsilon(Z_4) = (Z_4')\inv.
\end{equation}
Here, the images are understood as elements of the DAHA $\mathcal{H}' \coloneqq \mathcal{H}_{q\inv,\bm{\sigma}}$. At the classical level $q = 1$, this induces an isomorphism of the corresponding character varieties, given by
\begin{equation}\label{eq: duality on CM space}
\mathcal{E} : (A_1,A_2,A_3,A_4) \mapsto (A_1',A_2',A_3',A_4') \coloneqq (A_3\inv, A_2\inv, A_1\inv, A_4\inv),
\end{equation}
or, in the alternative form from \S\ref{subsec: alternative form}, given by
\begin{equation}\label{eq: alternative form duality on CM space}
\mathcal{E} : (X,T,Y,v,w) \mapsto (X',Y',T',v',w') \coloneqq (Y,X,T\inv,v,w).
\end{equation}

Let us use $\varepsilon'$ and $\mathcal{E}'$ for the duality maps applied to the DAHA and character variety, respectively, for $\mathcal{H}'$. The maps \eqref{eq: dual of Zi} and \eqref{eq: duality on CM space} are involutions in the sense $\varepsilon' \circ \varepsilon = \mathcal{E}' \circ \mathcal{E} = \id$. Obviously, we have $\mathcal{E}' \circ \Phi = \Phi \circ \varepsilon$, where $\Phi$ is the EGO map from Definition \ref{def: EGO map}. We can use $\mathcal{E}$ to construct a second coordinate chart on $\mathcal{M}_n$ by transferring coordinates from the corresponding variety $\mathcal{M}_n'$. In this chart, the matrix $Y = A_3\inv A_2\inv = A_4A_1$ is put into diagonal form
\begin{equation}\label{eq: diagonal Y}
Y = \diag(y_1,\dots,y_n,y_1\inv,\dots,y_n\inv),
\end{equation}
so $y_1,\dots,y_n$ give $n$ of the coordinates, with the remaining $n$ coordinates read-off of $A_2$ or $A_3$.

\section{The Isomorphism}\label{sec: isomorphism}

Proofs of our main results can now be obtained by repeating the arguments from \cite{oblomkovJun04}. We will briefly outline them here, but the reader should consult \textit{op. cit.} for the full details. Throughout this section, $\mathcal{H} = \mathcal{H}_{1,\bm{\tau}}$ denotes the DAHA at the classical level $q = 1$, and $\mathcal{Z} \coloneqq Z(\mathcal{H})$ denotes its centre. The DAHA is of type $C^\vee C_n$ unless specified otherwise.

First, we have the following general result due to Oblomkov.

\begin{theorem}[{\cite[Theorem 5.1]{oblomkovJun04}}]\label{thrm: spherical subalgebra}
For any DAHA $\mathcal{H}$, the following are true:
\begin{enumerate}[label=\rom]
	\item The spherical subalgebra $e_\tau\mathcal{H}e_\tau$ is commutative.
	\item The variety $M = \Spec(e_\tau\mathcal{H}e_\tau)$ is irreducible, normal and Cohen-Macaulay.
	\item The right $e_\tau\mathcal{H}e_\tau$-module $\mathcal{H}e_\tau$ is Cohen-Macaulay.
	\item The left $\mathcal{H}$-action on $\mathcal{H}e_\tau$ induces an algebra isomorphism $\mathcal{H} \cong \End_{e_\tau\mathcal{H}e_\tau}(\mathcal{H}e_\tau)$.
	\item The map $\eta : z \mapsto ze_\tau$ is an isomorphism $\mathcal{Z} \xrightarrow{\sim} e_\tau\mathcal{H}e_\tau$. Hence, $M = \Spec(\mathcal{Z})$.
\end{enumerate}
\end{theorem}

Analogously to \cite[Lemma 5.1]{oblomkovJun04} (cf. Corollary \ref{cor: localisation of the DAHA by delta} above), we have
\begin{equation*}
\mathcal{Z}_{\delta(\bm{X})\delta_{\bm{\tau}}(\bm{X})} \cong \mathbb{C}[\bm{X}^{\pm1},\bm{P}^{\pm1}]_{\delta(\bm{X})\delta_{\bm{\tau}}(\bm{X})}^W.
\end{equation*}
Equivalently, we have an isomorphism $\Spec(\mathcal{Z})_{\delta(\bm{X})\delta_{\bm{\tau}}(\bm{X})} \cong \mathcal{U}_{\bm{\tau}}/W$. We obtain from this the fact
\begin{equation}\label{eq: Upsilon as a map from Spec(Z)}
\Upsilon : \Spec(\mathcal{Z})_{\delta(\bm{X})\delta_{\bm{\tau}}(\bm{X})} \xrightarrow{\sim} \mathcal{M}_n^{\bm{\tau}}
\end{equation}
by combining it with $\Upsilon$ \eqref{eq: Upsilon} and Proposition \ref{prop: Upsilon is isomorphic on a subset}. Consequently, we have a rational map
\begin{equation}\label{eq: rational map}
\Upsilon : \Spec(\mathcal{Z}) \dasharrow \mathcal{M}_n.
\end{equation}

\begin{theorem}[{cf. \cite[Theorem 6.1]{oblomkovJun04}}]\label{thrm: main result}
The rational map $\Upsilon$ \eqref{eq: rational map} is a regular isomorphism of algebraic varieties. In particular, $M = \Spec(\mathcal{Z}) = \Spec(e_\tau\mathcal{H}e_\tau)$ is smooth.
\end{theorem}

\begin{proof}
Using duality in the senses of \eqref{eq: dual of Zi} and \eqref{eq: duality on CM space}, we obtain from \eqref{eq: Upsilon as a map from Spec(Z)} an isomorphism
\begin{equation*}
\Upsilon' : \Spec(\mathcal{Z})_{\delta(\bm{Y})\delta_{\bm{\sigma}}(\bm{Y})} \xrightarrow{\sim} \mathcal{M}_n^{\bm{\sigma}}
\end{equation*}
for a suitable open subset $\mathcal{M}_n^{\bm{\sigma}} \subseteq \mathcal{M}_n$. The maps $\Upsilon$ and $\Upsilon'$ agree on the intersection
\begin{equation}\label{eq: intersection of charts}
\Spec(\mathcal{Z})_{\delta(\bm{X})\delta_{\bm{\tau}}(\bm{X})} \cap \Spec(\mathcal{Z})_{\delta(\bm{Y})\delta_{\bm{\sigma}}(\bm{Y})}.
\end{equation}
Indeed, for a point $z$ in \eqref{eq: intersection of charts}, the corresponding $\mathcal{H}$-modules \eqref{eq: induced module V} are completely determined by the one-dimensional character $\chi_z$ representing this point $z$. As a result, $\Upsilon$ is regular on \eqref{eq: intersection of charts}. Using the Basic Representation, it is not too hard to see that the hypersurfaces $\delta(\bm{X})\delta_{\bm{\tau}}(\bm{X}) = 0$ and $\delta(\bm{Y})\delta_{\bm{\sigma}}(\bm{Y}) = 0$ are transversal in $\Spec(\mathcal{Z})$ (in fact, this is true for any hypersurfaces $f(\bm{X}) = 0$ and $g(\bm{Y}) = 0$ with $f \in \mathbb{C}[\bm{X}^{\pm1}]^W$ and $g \in \mathbb{C}[\bm{Y}^{\pm1}]^W$). It then follows that $\Upsilon$ is regular everywhere except a subset of codimension two. But the normality of $\Spec(\mathcal{Z})$ (Theorem \ref{thrm: spherical subalgebra}(ii)) implies that $\Upsilon$ is regular everywhere. After that, using the fact $\mathcal{M}_n$ is smooth, irreducible and that $\Upsilon\inv$ is regular away from a subset of codimension two, one sees that $\Upsilon\inv$ is also regular everywhere.
\end{proof}

As an immediate consequence, by the same arguments as in \cite{oblomkovJun04}, we arrive at the following.

\begin{corollary}[{cf. \cite[Corollaries 6.1 and 6.2]{oblomkovJun04}}]\phantomsection\label{cor: main result}
\begin{enumerate}[label=\rom]
	\item $\mathcal{H}e_\tau$ is a projective $e_\tau\mathcal{H}e_\tau$-module.
	\item $\mathcal{H} = \End(E)$, where $E$ is a vector bundle over $\Spec(\mathcal{Z})$, i.e. $\mathcal{H}$ is an Azumaya algebra.
	\item Every irreducible representation of $\mathcal{H}$ is of the form $V_z = \mathcal{H}e_\tau \otimes_{e_\tau\mathcal{H}e_\tau} \chi_z$ for $z \in \Spec(\mathcal{Z})$.
	\item $V_z$ has dimension $2^nn!$, and is a regular representation of the finite Hecke algebra $H_n$.
\end{enumerate}
\end{corollary}

\subsection{Quantisation}\label{subsec: quantisation}

Because our calculations of the matrices $A_i$ in \S\ref{subsec: coordinate chart} were performed at the quantum level, we can generalise the isomorphism \eqref{eq: main result} to $q \neq 1$. To begin with, we assume $q = 1$ and work with $\mathcal{M}_n$ in the alternative form introduced in \S\ref{subsec: alternative form} with relations \eqref{eq: T relation}--\eqref{eq: wv relation}.

\begin{lemma}\label{lem: algebra of regular functions at q = 1}
The algebra $\mathbb{C}[\mathcal{M}_n]$ of regular functions on the Calogero--Moser space is generated by $wa(X,Y,T)v$, with arbitrary non-commutative polynomials $a \in \mathbb{C}\inner{X^{\pm1},Y^{\pm1},T^{\pm1}}$.
\end{lemma}

We now allow $q$ to be arbitrary and consider the $\mathscr{D}_q$-valued matrices $A_i$ defined by Proposition \ref{prop: matrices of Ai}. We also take $X = A_1A_2$, $Y = A_4A_1$, $T = A_2$, $\mathbf{v}$ and $\mathbf{w}$ as in Lemma \ref{lem: action of Hecke symmetriser on M'}.

\begin{proposition}\label{prop: quantised main result}
The elements $\mathbf{w}b(X,Y,T)\mathbf{v}$ form a subalgebra of $\mathscr{D}_q^W$ isomorphic to $e_\tau\mathcal{H}_{q,\bm{\tau}}e_\tau$, with arbitrary non-commutative polynomials $b \in \mathbb{C}\inner{X^{\pm1},Y^{\pm1},T^{\pm1}}$.
\end{proposition}

\begin{proof}
Recall the elements $Z_i \in \mathcal{H}_{q,\bm{\tau}}$ from \eqref{eq: Zi}. For $a \in \mathbb{C}\inner{Z_1^{\pm1},Z_2^{\pm1},Z_3^{\pm1},Z_4^{\pm1}}$, the elements
\begin{equation*}
e_\tau a(Z_1,Z_2,Z_3,Z_4)e_\tau
\end{equation*}
clearly belong to the spherical subalgebra $e_\tau\mathcal{H}_{q,\bm{\tau}}e_\tau$. Moreover, the relation \eqref{eq: Z4 relation} implies that these elements form a subalgebra of $e_\tau\mathcal{H}_{q,\bm{\tau}}e_\tau$, which we denote as $A_{q,\bm{\tau}}$. The matrix representation $\pi$ \eqref{eq: matrix representation}, combined with Lemmas \ref{lem: Ai preserve M'} and \ref{lem: action of Hecke symmetriser on M'}, defines an isomorphic subalgebra $B_{q,\bm{\tau}} \subseteq \mathscr{D}_q^W$:
\begin{equation*}
A_{q,\bm{\tau}} \xrightarrow{\sim} B_{q,\bm{\tau}}, \qquad e_\tau a(Z_1,Z_2,Z_3,Z_4)e_\tau \mapsto \gamma\inv\mathbf{w}a(A_1,A_2,A_3,A_4)\mathbf{v},
\end{equation*}
where $\gamma$ is the constant \eqref{eq: gamma}. At the classical level $q = 1$, we have $B_{1,\bm{\tau}} \cong e_\tau\mathcal{H}_{1,\bm{\tau}}e_\tau$ by Lemma \ref{lem: algebra of regular functions at q = 1} combined with Theorem \ref{thrm: main result}. Hence, they remain isomorphic for arbitrary $q$. Finally, any element of the form $\mathbf{w}a(A_1,A_2,A_3,A_4)\mathbf{v}$ can be transformed into $\mathbf{w}b(X,Y,T)\mathbf{v}$ by using the relation $q^{1/2}A_1A_2A_3A_4 = 1$ from Lemma \ref{lem: product of DAHA elements is the identity}.
\end{proof}

\begin{corollary}\label{cor: spherical subaglebra generators}
The spherical subalgebra $e_\tau\mathcal{H}_{q,\bm{\tau}}e_\tau$ is, as an algebra, generated by $e_\tau X_1^aY_1^be_\tau$ and $e_\tau X_1^aY_1^bT_0^\vee e_\tau$ with $a,b \in \mathbb{Z}$. Similarly, its copy in $\mathscr{D}_q^W$ is generated by $\mathbf{w}X^aY^b\mathbf{v}$ and $\mathbf{w}X^aY^bT\mathbf{v}$ with $a,b \in \mathbb{Z}$.
\end{corollary}

\begin{remark}\label{rem: quantum Hamiltonian reduction}
It would be interesting to quantise the isomorphism \eqref{eq: main result} using the framework of quantum Hamiltonian reduction, in the style of \cite{etingofloktevoblomkovrybnikov,jordan,wen}.
\end{remark}

\begin{remark}\label{rem: quiver map}
Recall that $\mathcal{M}_n$ can be viewed as a multiplicative quiver variety, by Corollary \ref{cor: CM space isomorphic to quiver variety}. In this case, one can define a map analogous to $\Theta^\text{Quiver}$ in \cite[(1.6.3)]{etingofganginzburgoblomkov}. This can be used to relate simple $\mathcal{H}$-modules to ideals of the quantum spherical $C^\vee C_1$ DAHA, similarly to \cite[Theorem 7]{berestchalykheshmatov}. This will be discussed elsewhere.
\end{remark}

\section{Application to the Trigonometric van Diejen System}\label{sec: dynamics}

In this section, we apply our results to study the trigonometric van Diejen system \cite{vandiejen}. First, recall that the family $e_\tau\mathcal{H}_{q,\bm{\tau}}e_\tau$ provides a non-commutative deformation of the commutative algebra $e_\tau\mathcal{H}e_\tau \cong \mathcal{Z}$, by which the latter inherits a Poisson bracket. We therefore may, and will, use the isomorphism \eqref{eq: main result} to induce a Poisson bracket $\{\,\cdot\,,\,\cdot\,\}$ on the character variety $\mathcal{M}_n$. As an immediate consequence, we have the following result.

\begin{proposition}\phantomsection\label{prop: Poisson bracket on CM space}
\begin{enumerate}[label=\rom]
	\item The coordinates $X_i,P_i$ on $\mathcal{M}_n$ are log-canonical, that is
	\begin{equation*}
	\{P_i,X_j\} = \delta_{ij}P_iX_j, \qquad \{P_i,P_j\} = \{X_i,X_j\} = 0.
	\end{equation*}
	\item The duality map $\mathcal{E}$ \eqref{eq: duality on CM space} is a Poisson anti-automorphism, so $\mathcal{E}_\ast\{\,\cdot\,,\,\cdot\,\}_{\mathcal{M}_n} = -\{\,\cdot\,,\,\cdot\,\}_{\mathcal{M}_n'}$.
	\item The Poisson bracket on $\mathcal{M}_n$ is non-degenerate, so $\mathcal{M}_n$ is symplectic.
\end{enumerate}
\end{proposition}

\begin{proof}
(i) This follows because the coordinates $X_i,P_i$ correspond to the generators of the algebra $\mathbb{C}[\bm{X}^{\pm1},\bm{P}^{\pm1}]$, whose Poisson bracket comes from the quantised algebra $\mathbb{C}_q[\bm{X}^{\pm1},\bm{P}^{\pm1}]$.

(ii) The duality map $\mathcal{E}$ is induced by the algebra automorphism $\varepsilon$. It therefore respects the Poisson structure on the spherical subalgebra, with the change of sign for the bracket caused by the fact that $\varepsilon$ interchanges $q \leftrightarrow q\inv$.

(iii) For the non-degeneracy, we observe in (i) that the bracket is log-canonical and thus non-degenerate in both coordinate charts. Therefore, the bracket is non-degenerate except on a subset of codimension two, and hence non-degenerate globally.
\end{proof}

We now work in the first chart $(X_i,P_i)$ and consider the functions $h_k \coloneqq \tr X^k$ with $X = A_1A_2$,
\begin{equation}\label{eq: Hamiltonian trace of powers of X}
h_k = \sum_{i=1}^n (X_i^k + X_i^{-k}).
\end{equation}
 It is clear that $\{h_k,h_l\} = 0$ for all $k,l = 1,\dots,n$. Furthermore, we have
\begin{equation}\label{eq: XP dynamics}
\{X_i,h_k\} = 0, \qquad \{P_i,h_k\} = kP_i(X_i^k - X_i^{-k}).
\end{equation}
This means that the Hamiltonian dynamics governed by $h_k$ is separated in $(\bm{X},\bm{P})$ coordinates:
\begin{equation*}
X_i(t) = X_i(0), \qquad P_i(t) = e^{kt(X_i^k - X_i^{-k})}P_i(0).
\end{equation*}

We describe the corresponding dynamics in invariant terms, globally on $\mathcal{M}_n$. It is convenient to do so on the representation variety \eqref{eq: DS problem} associated to the conjugacy classes defined by \eqref{eq: our eigendata},
\begin{equation*}
\mathfrak{R}_n \coloneqq \mathfrak{R}_{0,4} = \{A_i \in C_i : A_1A_2A_3A_4 = \Id_{2n}\}.
\end{equation*}

Recall $\mathcal{M}_n = \mathfrak{R}_n\GIT\GL_{2n}(\mathbb{C})$. Introduce the following $\GL_{2n}(\mathbb{C})$-invariant vector field on $\mathfrak{R}_n$:
\begin{equation}\label{eq: dynamics on Rn in the first chart}
\dot{A}_1 = -k(A_1X^k - X^kA_1), \qquad \dot{A}_2 = -k(A_2X^k - X^kA_2), \qquad \dot{A}_3 = 0, \qquad \dot{A}_4 = 0.
\end{equation}
This can be easily integrated, giving $X = A_1A_2$ constant and
\begin{equation*}
A_1(t) = e^{ktX^k}A_1(0)e^{-ktX^k}, \quad A_2(t) = e^{ktX^k}A_2(0)e^{-ktX^k}, \quad A_3(t) = A_3(0), \quad A_4(t) = A_4(0).
\end{equation*}

\begin{proposition}\label{prop: dynamics on CM space in the first chart}
The Hamiltonian dynamics on $\mathcal{M}_n$ governed by $h_k = \tr X^k$ can be obtained by projecting the dynamics \eqref{eq: dynamics on Rn in the first chart} onto $\mathcal{M}_n$. The dynamics is complete on $\mathcal{M}_n$.
\end{proposition}

\begin{proof}
This can be confirmed by a straightforward calculation in coordinates, using \eqref{eq: XP dynamics}. By analytic continuation, the result is valid globally on $\mathcal{M}_n$. Because the dynamics is obviously complete on the auxiliary space $\mathfrak{R}_n$, it is complete on $\mathcal{M}_n$.
\end{proof}

Applying the duality map $\mathcal{E}$, one can interchange the roles of $X = A_1A_2$ and $Y = A_4A_1$. The latter matrix $Y$ has been identified in \cite[Corollary 4.4]{chalykh} as a Lax matrix for the van Diejen system. In particular, its conserved quantities are given by $H_k \coloneqq \tr Y^k$, so we will take them as Hamiltonians of this system. The corresponding dynamics on $\mathfrak{R}_n$ takes the form
\begin{equation}\label{eq: dynamics on Rn in the second chart}
\dot{A}_1 = 0, \qquad \dot{A}_2 = -k(Y^kA_2 - A_2Y^k), \qquad \dot{A}_3 = -k(Y^kA_3 - A_3Y^k), \qquad \dot{A}_4 = 0,
\end{equation}
which integrates to give $Y = A_4A_1$ constant and
\begin{equation*}
A_1(t) = A_1(0), \quad A_2(t) = e^{-ktY^k}A_2(0)e^{ktY^k}, \quad A_3(t) = e^{-ktY^k}A_3(0)e^{ktY^k}, \quad A_4(t) = A_4(0).
\end{equation*}

\begin{theorem}\label{thrm: dynamics on CM space in the second chart}
The Hamiltonian dynamics on $\mathcal{M}_n$ governed by $H_k = \tr Y^k$ can be obtained by projecting the dynamics \eqref{eq: dynamics on Rn in the second chart} onto $\mathcal{M}_n$. The dynamics is complete on $\mathcal{M}_n$.
\end{theorem}

By Lemma \ref{lem: eigenvalues of X are pairewise reciprocal}, the matrix $X(t) = A_1(t)A_2(t)$ has pairwise-reciprocal eigenvalues $x_i,x_i\inv$ for any $t$, giving the positions of the particles at time $t$.

The second coordinate chart on $\mathcal{M}_n$ provides the action-angle variables for the van Diejen system. The action variables are determined by the eigenvalues of $Y$, and the angle variables are the dual counterparts of $P_i$. Correspondingly, the functions $h_k = \tr X^k$ in terms of these action-angle coordinates assume the form of the van Diejen Hamiltonians with dual parameters. This picture is analogous to the Ruijsenaars duality well-known in the $\GL_n$-case \cite{ruijsenaars,fockgorskynekrasovrubtsov,feherklimcik}, and it is a non-trivial manifestation of the duality for DAHAs. In a special limiting case of the five-parameter van Diejen system, such duality was established in \cite{fehermarshall17,fehermarshall19}.

\begin{remark}\label{rem: Lax represenation of van Diejen dynamics}
For  a two-parameter case $k_0 = u_0 = u_n = 1$, the van Diejen system has been thoroughly studied by Pusztai and G\"orbe \cite{pusztaigorbe} (cf. \cite{avanrollet} for some results in a one-parameter case). We explain how their results fit into ours: the relations \eqref{eq: T relation}--\eqref{eq: YTX relation} simplify to
\begin{equation*}
\begin{gathered}
T = T\inv, \qquad XT\inv = TX\inv, \qquad YT = T\inv Y\inv,\\
(tYX - t\inv XY)T = (k_nt\inv - k_n\inv t)\Id_V + (t-t\inv)vw.
\end{gathered}
\end{equation*}
With this, $X(t)$ can be rearranged to $X(t) = X(0)e^{kt(Y^k-Y^{-k})}$, which matches the formulae in \textit{op. cit.} (cf. \cite[(4.113)]{pusztaigorbe} for example) up to the change of notation $X = e^{2\Lambda}$, $Y = L$, $T = C$.
\end{remark}

\subsection{Interpretation via Hamiltonian Reduction}\label{subsec: reduction}

A character variety of a Riemann surface can be seen as a result of (infinite-dimensional) Hamiltonian reduction, performed on the moduli space of smooth connections \cite{atiyahbott}, and so it has a canonical Poisson structure, cf. \cite{goldman84,goldman86}. Fock and Rosly \cite{fockrosly} explained how to obtain the same space by a finite-dimensional reduction, modelling flat connections by combinatorial connections on graphs embedded into the surface. Applying this to the one-punctured torus, they obtained the variety $CM_\tau$ appearing in Theorem \ref{thrm: Oblomkov isomorphism} and were able to interpret it as a completed phase space for the Ruijsenaars--Schneider system, see \cite[Appendix]{fockrosly} (the observation that the Ruijsenaars--Schneider system naturally arises on the moduli space of flat connections on the one-punctured torus goes back to \cite{gorskynekrasov}). Carrying out the same approach for the four-punctured sphere would be technically difficult, but we can fortunately arrive at a similar interpretation without much hard work, thanks to the results at hand. Below, we freely use the terminology and notation from \cite{fockrosly}, so the reader should consult that paper for the details.

\begin{figure}[H]
\centering
\begin{tikzpicture}[decoration={markings,mark=at position 0.5 with {\arrow{>}}}]
\filldraw (0,0) circle (2pt);
\draw[postaction={decorate}] (0,0) to[out=135,in=45,loop,distance=6cm] (0,0);
\draw[postaction={decorate}] (0,0) to[out=135,in=45,loop,distance=4.5cm] (0,0);
\draw[postaction={decorate}] (0,0) to[out=135,in=45,loop,distance=3cm] (0,0);
\node at (-0.1,3.5) {\footnotesize$a$};
\node at (-0.1,2.7) {\footnotesize$b$};
\node at (-0.1,1.85) {\footnotesize$c$};
\end{tikzpicture}
\caption{The graph corresponding to the four-punctured sphere.}
\label{fig: four-punctured sphere graph}
\end{figure}
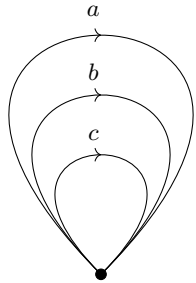

Let us consider the space of $G$-valued graph connections on the graph in Figure \ref{fig: four-punctured sphere graph}, representing the sphere with four holes (punctures). It has three edges, to each of which we associate an element of $G$. Hence, the moduli space of graph connections in this case is $\mathcal{A}^l = G \times G \times G$. The graph has one vertex, so the gauge group $\mathcal{G}^l = G$ acts on $\mathcal{A}^l$ by simultaneous conjugation.

According to \cite{fockrosly}, the choice of a Poisson structure on $\mathcal{A}^l$ is based on a choice of a classical $r$-matrix. We work with the group $G = \GL_N(\mathbb{C})$ for $N = 2n$ and choose
\begin{equation*}
r = \sum_{i<j}^N E_{ij} \otimes E_{ji} + \frac{1}{2}\sum_{i=1}^N E_{ii} \otimes E_{ii}.
\end{equation*}

\begin{remark}\label{rem: r-matrix decomposition}
We can express $r = r_a + t$ as the sum of a skew-symmetric and symmetric part, for
\begin{equation*}
r_a = \frac{1}{2}(r - r_{21}) \qquad\text{and}\qquad t = \frac{1}{2}(r + r_{21}).
\end{equation*}
\end{remark}

This defines a Poisson bivector on $\mathcal{A}^l$, by \cite[(16)]{fockrosly}, which then induces a Poisson bracket on $\mathcal{A}^l/\mathcal{G}^l$. By \cite[Proposition 5]{fockrosly}, the resulting Poisson manifold is isomorphic to the moduli space $\mathcal{M}$ of flat connections on the four-punctured sphere, with the Atiyah--Bott Poisson structure. Symplectic leaves correspond to fixing the conjugacy classes of holonomies around the punctures, so we choose the conjugacy classes $C_i$ defined by \eqref{eq: our eigendata}, resulting in our character variety of the Calogero--Moser space $\mathcal{M}_n$. In terms of $(\mathbf{A},\mathbf{B},\mathbf{C}) \in \mathcal{A}^l$, this means imposing the constraints
\begin{equation*}
\mathbf{A} \in C_1, \qquad \mathbf{A}\inv\mathbf{B} \in C_2, \qquad \mathbf{B}\inv\mathbf{C} \in C_3, \qquad \mathbf{C}\inv \in C_4.
\end{equation*}
Therefore, a priori, we have two Poisson brackets on $\mathcal{M}_n$: the Atiyah--Bott bracket $\{\,\cdot\,,\,\cdot\,\}_\text{AB}$ and the bracket coming from the DAHA $\{\,\cdot\,,\,\cdot\,\}_\text{DAHA}$. We claim that these brackets are the same. Before the proof, we write the Poisson brackets on $\mathcal{A}^l$, which are analogues of \cite[(A2)--(A4)]{fockrosly}:
\begin{align}
\{\mathbf{A},\mathbf{A}\} &= r_a(\mathbf{A} \otimes \mathbf{A}) + (\mathbf{A} \otimes \mathbf{A})r_a + (1 \otimes \mathbf{A})r_{21}(\mathbf{A} \otimes 1) - (\mathbf{A} \otimes 1)r(1 \otimes \mathbf{A}),\label{eq: FR bracekt A with A}\\
\{\mathbf{A},\mathbf{B}\} &= r(\mathbf{A} \otimes \mathbf{B}) - (\mathbf{A} \otimes \mathbf{B})r_{21} + (1 \otimes \mathbf{B})r_{21}(\mathbf{A} \otimes 1) - (\mathbf{A} \otimes 1)r(1 \otimes \mathbf{B}),\label{eq: FR bracket A with B}
\end{align}
with $\{\mathbf{B},\mathbf{B}\}$, $\{\mathbf{C},\mathbf{C}\}$ completely similar to \eqref{eq: FR bracekt A with A}, and $\{\mathbf{A},\mathbf{C}\}$, $\{\mathbf{B},\mathbf{C}\}$ completely similar to \eqref{eq: FR bracket A with B}. 

\begin{proposition}\label{prop: our isomorphism is a Poisson map}
The isomorphism \eqref{eq: main result} is a Poisson map, that is it identifies the natural Poisson bracket on the spherical subalgebra with the Atiyah--Bott bracket on the character variety.
\end{proposition}

\begin{proof}
Working on $\mathcal{A}^l$ with the Fock--Rosly bracket, we take $\tr \mathbf{B}^k$ and calculate its brackets with $\mathbf{A}$, $\mathbf{B}$, $\mathbf{C}$. This is similar to the way \cite[(A6)--(A8)]{fockrosly} are derived, and the result is
\begin{equation}
\{\mathbf{A},\tr\mathbf{B}^k\} = -k(\mathbf{A}\mathbf{B}^k - \mathbf{B}^k\mathbf{A}), \qquad \{\mathbf{B},\tr\mathbf{B}^k\} = 0, \qquad \{\mathbf{C},\tr\mathbf{B}^k\} = 0.
\end{equation}
Upon the identification $\mathbf{A} = A_1$, $\mathbf{B} = A_1A_2 = X$ and $\mathbf{C} = A_4\inv$, the vector field $\{\,\cdot\,,\tr\mathbf{B}^k\}$ is the same as in \eqref{eq: dynamics on Rn in the first chart}. Projecting onto $\mathcal{M}_n \cong \mathcal{A}^l/\mathcal{G}^l$, it becomes clear that $\tr X^k$ defines the same vector field with respect to each of the Poisson brackets, which is to say
\begin{equation*}
\{\,\cdot\,,\tr X^k\}_\text{AB} = \{\,\cdot\,,\tr X^k\}_\text{DAHA}.
\end{equation*}
Similarly, for $Y = \mathbf{C}\inv \mathbf{A}$, we calculate brackets between $\tr Y^k$ and $\mathbf{A}$, $\mathbf{B}$, $\mathbf{C}$, and notice
\begin{equation*}
\{\,\cdot\,,\tr Y^k\}_\text{AB} = \{\,\cdot\,,\tr Y^k\}_\text{DAHA}.
\end{equation*}
We see now that $\tr X^k$ and $\tr Y^k$ are in the kernel of $\{\,\cdot\,,\,\cdot\,\} \coloneqq \{\,\cdot\,,\,\cdot\,\}_\text{AB} - \{\,\cdot\,,\,\cdot\,\}_\text{DAHA}$. But the $2n$ functions $(\tr X^k,\tr Y^k)_{k=1,\dots,n}$ are functionally independent generically on $\mathcal{M}_n$, so $\{\,\cdot\,,\,\cdot\,\} \equiv 0$.
\end{proof}

\begin{corollary}\label{cor: spherical subalgebra equipped with Atiyah-Bott bracket}
The coordinates $X_i,P_i$ on $\mathcal{M}_n$ are log-canonical with respect to the Atiyah--Bott bracket. The spherical subalgebra $e_\tau\mathcal{H}_{q,\bm{\tau}}e_\tau$ provides a quantisation of the character variety $\mathcal{M}_n$, equipped with the Atiyah--Bott bracket.
\end{corollary}

We conclude that the symplectic leaf $\mathcal{M}_n$ of the moduli space $\mathcal{M}$ of flat $\GL_{2n}(\mathbb{C})$-connections on the four-punctured sphere serves as a completed phase space for the van Diejen system.

\begin{remark}\label{rem: quasi-Hamiltonian reduction}
The variety $\mathcal{M}_n$, as a multiplicative quiver variety, can also be obtained by quasi-Hamiltonian reduction by the result of Van den Bergh \cite{vandenbergh}; see also \cite{massuyeauturaev}. We expect the resulting Poisson bracket to coincide with the Atiyah--Bott bracket.
\end{remark}

\begin{remark}\label{rem: spin version}
The interpretation of the van Diejen system within the moduli space of flat connections on a punctured sphere leads to a symplectic action of (a subgroup of) the mapping class group on the family of these systems and corresponding spaces $\mathcal{M}_n$. This also suggests that there should be a geometric way of constructing and studying a spin version of the system, using the approach of \cite{arthamonovreshetikhin,chalykhfairon}.
\end{remark}

\phantomsection\section*{References}
\addcontentsline{toc}{section}{References}
\bibliographystyle{alphaurl}
\bibliography{References}

\bigskip\bigskip\bigskip
\begin{tabular}{@{}l@{}}%
\textsc{Oleg Chalykh}: School of Mathematics, University of Leeds, United Kingdom\\
\textit{e-mail address}: \href{mailto:O.Chalykh@leeds.ac.uk}{\texttt{O.Chalykh@leeds.ac.uk}}\\\\
\textsc{Bradley Ryan}: School of Mathematics, University of Leeds, United Kingdom\\
\textit{e-mail address}: \href{mailto:B.J.Ryan@leeds.ac.uk}{\texttt{B.J.Ryan@leeds.ac.uk}}
\end{tabular}
\end{document}

%% file: Preamble.tex
\usepackage{adjustbox}
\usepackage{amsbsy}
\usepackage{amsfonts}
\usepackage{amsmath}
\usepackage{amssymb}
\usepackage{amsthm}
\usepackage{authblk}
\usepackage{bm}
\usepackage{bbm}
\usepackage{bibentry}
\usepackage{blkarray}
\usepackage{booktabs}
\usepackage{caption}
\usepackage{cellspace}
\usepackage{centernot}
\usepackage{dynkin-diagrams}
\usepackage{enumitem}
\usepackage{float}
\usepackage{graphicx}
\usepackage{lastpage}
\usepackage{listings}
\usepackage{lscape}
\usepackage{mathtools}
\usepackage{mathrsfs}
\usepackage{multicol}
\usepackage{multirow}
\usepackage{parskip}
\usepackage{physics}
\usepackage{setspace}
\usepackage{siunitx}
\usepackage{stackengine}
\usepackage{tabto}
\usepackage{tabularx}
\usepackage{tcolorbox}
\usepackage{titlesec}
\usepackage{tikz}
\usepackage{tikz-cd}
\usepackage{tikzpagenodes}
\usepackage{trimclip}
\usepackage{upgreek}
\usepackage{url}
\usepackage{xparse}
\usepackage{xpatch}
\usepackage{xr}
\usepackage[makeroom]{cancel}
\usepackage[sort,numbers]{natbib}
\usepackage[left=2.75cm,right=2.75cm,top=3cm,bottom=3cm]{geometry}
\usepackage[hidelinks]{hyperref}
\usepackage[utf8]{inputenc}
\usepackage[super]{nth}
\usepackage[figuresleft]{rotating}
\usepackage[labelformat=simple]{subcaption}

\allowdisplaybreaks

\titleformat{\section}{\normalfont\sffamily\Large\bfseries}{\thesection}{1em}{}
\titleformat{\subsection}{\normalfont\sffamily\large\bfseries}{\thesubsection}{1em}{}
\titleformat{\subsubsection}{\normalfont\sffamily\bfseries}{\thesubsubsection}{1em}{}
\hypersetup{colorlinks={true},linkcolor={purple},citecolor={purple},urlcolor={purple}}
\newcommand{\rom}{\normalshape(\roman*)}

\renewcommand{\emph}[1]{\textbf{\sffamily#1}}
\numberwithin{equation}{section}
\newtheoremstyle{mytheorem}
{} 
{} 
{\itshape} 
{} 
{\sffamily\bfseries} 
{\sffamily\bfseries} 
{ } 
{\thmname{#1}\thmnumber{ #2}\thmnote{ \normalfont(#3)}}

\newtheoremstyle{mydefinition}
{} 
{} 
{} 
{} 
{\sffamily\bfseries} 
{\sffamily\bfseries} 
{ } 
{\thmname{#1}\thmnumber{ #2}\thmnote{ \normalfont(#3)}}

\theoremstyle{mytheorem}
\newtheorem{theorem}{Theorem}[section]
\newtheorem{lemma}[theorem]{Lemma}
\newtheorem{proposition}[theorem]{Proposition}
\newtheorem{corollary}[theorem]{Corollary}

\newtheorem{problem}[theorem]{Problem}
\newtheorem*{theorem*}{Theorem}
\newtheorem*{lemma*}{Lemma}
\newtheorem*{proposition*}{Proposition}
\newtheorem*{corollary*}{Corollary}
\newtheorem*{conjecture*}{Conjecture}
\newtheorem*{problem*}{Problem}

\theoremstyle{mydefinition}
\newtheorem{definition}[theorem]{Definition}
\newtheorem{remark}[theorem]{Remark}

\newtheorem{notation}[theorem]{Notation}

\newtheorem*{definition*}{Definition}
\newtheorem*{remark*}{Remark}
\newtheorem*{example*}{Example}
\newtheorem*{notation*}{Notation}
\newtheorem*{reminder*}{Reminder}

\usetikzlibrary{decorations.markings}
\usetikzlibrary{patterns}
\usetikzlibrary{intersections}
\usetikzlibrary{shapes.geometric}

\DeclareMathOperator{\GL}{GL}

\DeclareMathOperator{\SU}{SU}
\DeclareMathOperator{\id}{id}

\DeclareMathOperator{\PGL}{PGL}

\DeclareMathOperator{\Ind}{Ind}

\DeclareMathOperator{\Hom}{Hom}
\DeclareMathOperator{\End}{End}

\DeclareMathOperator{\Ext}{Ext}

\DeclareMathOperator{\Mat}{Mat}

\DeclareMathOperator{\image}{im}

\DeclareMathOperator{\Hilb}{Hilb}

\DeclareMathOperator{\Spec}{Spec}
\DeclareMathOperator{\diag}{diag}

\DeclareMathOperator{\Irrep}{Irrep}

\DeclareMathOperator{\s}{\mathfrak{s}}

\DeclareMathOperator{\Rep}{Rep}

\newcommand{\defn}[1]{{\normalfont\sffamily#1}}

\newcommand{\act}{\mathbin{\vcenter{\hbox{\scalebox{0.75}{$\bullet$}}}}}

\newcommand{\inv}{\ensuremath{^{-1}}}

\newcommand{\Id}{\mathbbm{1}}
\newcommand{\say}[1]{``#1"}

\newcommand{\inner}[1]{\left\langle#1\right\rangle}

\newcommand{\GIT}{\mathbin{\mathchoice{/\mkern-6mu/}{/\mkern-6mu/}{/\mkern-5mu/}{/\mkern-5mu/}}}

\def\diamondprod#1#2{\ooalign{\hidewidth$#1\middiamond$\hidewidth\cr$#1\prod$\cr}}
\def\middiamond{\mathchoice
	{\vcenter{\hbox{\phantom{$\displaystyle\prod$}$\diamond$\phantom{$\displaystyle\prod$}}}}
	{\vcenter{\hbox{\phantom{$\prod$}$\diamond$\phantom{$\prod$}}}}
	{\vcenter{\hbox{\phantom{$\scriptstyle\prod$}$\scriptstyle\diamond$\phantom{$\scriptstyle\prod$}}}}
	{\vcenter{\hbox{\phantom{$\scriptscriptstyle\prod$}$\scriptscriptstyle\diamond$\phantom{$\scriptscriptstyle\prod$}}}}
}

\DeclareMathOperator*\dprod{\mathpalette\diamondprod\relax}

\newcolumntype{L}{>{\raggedright\arraybackslash}X}
\newcolumntype{R}{>{\raggedleft\arraybackslash}X}
\newcolumntype{C}{>{\centering\arraybackslash}X}

\newlength{\PRLlen}
\newcommand*{\trule}[1]{\settowidth{\PRLlen}{#1}\advance\PRLlen by -\textwidth\divide\PRLlen by -2\noindent\makebox[\the\PRLlen]{\resizebox{\the\PRLlen}{1pt}{$\blacktriangleleft$}}\raisebox{-.5ex}{#1}\makebox[\the\PRLlen]{\resizebox{\the\PRLlen}{1pt}{$\blacktriangleright$}}}

\let\originalleft\left
\let\originalright\right
\renewcommand{\left}{\mathopen{}\mathclose\bgroup\originalleft}
\renewcommand{\right}{\aftergroup\egroup\originalright}

\newcommand{\textover}[3][l]{\makebox[\widthof{#3}][#1]{#2}}

\newcommand{\nbracket}[2][]{\sbox0{\mathsurround=0pt$#1\{$}\sbox2{\{}\ifdim\ht0=\ht2\{\kern-.625\wd2 \{#2\}\kern-.625\wd2 \}\else\mathopen{#1\{\kern-.7\wd0 #1\{}#2\mathclose{#1\}\kern-.7\wd0 #1\}}\fi}

\newcount\dotcnt\newdimen\deltay
\def\Ddot#1#2(#3,#4,#5,#6){\deltay=#6\setbox1=\hbox to0pt{\smash{\dotcnt=1
			\kern#3\loop\raise\dotcnt\deltay\hbox to0pt{\hss#2}\kern#5\ifnum\dotcnt<#1
			\advance\dotcnt 1\repeat}\hss}\setbox2=\vtop{\box1}\ht2=#4\box2}

\newsavebox\myLogoBox
\savebox\myLogoBox{%
\begin{tikzpicture}[remember picture]
\begin{scope}[transform shape,scale=0.125]
\fill[black] (16,0) -- (16.1,0) -- (16.1,16.1) -- (0,16.1);
\fill[white] (8.45,0) -- (9.925,0) -- (9.925,4.65) -- (8.45,4.65);
\fill[white] (10.175,0) -- (12.845,0) -- (12.845,5.125) -- (10.175,5.125);
\draw[black,line width=0.2mm] (11.51,3.6675) circle (0.755cm);
\fill[white] (13.095,0) -- (14.57,0) -- (14.57,4.65) -- (13.095,4.65);
\fill[white] (8.775,4.9) -- (9.925,4.9) -- (9.925,5.35) -- (13.095,5.35) -- (13.095,4.9) -- (14.245,4.9) -- (14.245,5.875) -- (14.02,5.875) -- (14.02,9.385) -- (13.72,9.385) -- (13.72,11.175) -- (9.3,11.175) -- (9.3,9.385) -- (9,9.385) -- (9,5.875) -- (8.775,5.875);
\fill[black] (9.95,5.875) -- (10.2,5.875) -- (10.2,8.2) -- (10.75,8.2) -- (10.75,8.65) -- (9.95,8.65);
\fill[black] (11.11,5.875) -- (11.36,5.875) -- (11.36,8.2) -- (11.91,8.2) -- (11.91,8.65) -- (11.11,8.65);
\fill[black] (12.27,5.875) -- (12.52,5.875) -- (12.52,8.2) -- (13.07,8.2) -- (13.07,8.65) -- (12.27,8.65);
\fill[black] (9.515,9.29) -- (9.68,9.025) -- (13.34,9.025) -- (13.505,9.29);
\fill[black] (10.65,10.2) -- (10.8,10.2) -- (10.8,10.6) -- (11.01,10.6) -- (11.01,10.825) -- (10.65,10.825);
\fill[black] (11.33,10.2) -- (11.48,10.2) -- (11.48,10.6) -- (11.69,10.6) -- (11.69,10.825) -- (11.33,10.825);
\fill[black] (12.01,10.2) -- (12.16,10.2) -- (12.16,10.6) -- (12.37,10.6) -- (12.37,10.825) -- (12.01,10.825);
\fill[white] (9.85,11.45) -- (10,12.45) -- (13.02,12.45) -- (13.17,11.45);
\fill[white] (10.1,12.7) -- (12.92,12.7) -- (11.51,13.525);
\end{scope}
\end{tikzpicture}%
}
\makeatletter
\newcommand{\tpmod}[1]{{\@displayfalse\pmod{#1}}}
\AtBeginDocument{\xpatchcmd{\@thm}{\thm@headpunct{.}}{\thm@headpunct{}}{}{}}
\AtBeginDocument{\xpatchcmd{\proof}{\@addpunct{.}}{\normalfont\,\@addpunct{:}}{}{}}
\patchcmd{\@maketitle}{\LARGE}{\bf\LARGE\sffamily}{}{}
\patchcmd{\@maketitle}{\large\lineskip}{\Large\lineskip}{}{}
\renewcommand*{\@textcolor}[3]{\protect\leavevmode\begingroup\color#1{#2}#3\endgroup}
\newcommand{\leqnomode}{\tagsleft@true\let\veqno\@@leqno}
\newcommand{\reqnomode}{\tagsleft@false\let\veqno\@@eqno}
\newcommand{\tpitchfork}{\vbox{\baselineskip\z@skip\lineskip-.52ex\lineskiplimit\maxdimen\m@th\ialign{##\crcr\hidewidth\smash{$-$}\hidewidth\crcr$\pitchfork$\crcr}}}
\renewcommand{\twoheadrightarrow}{\mathrel{\text{\two@rightarrow}}}
\newcommand{\two@rightarrow}{\sbox0{$\m@th\rightarrow$}\smash{\rlap{\kern0.1\wd0\clipbox{{.3\width} {-\height} 0pt {-\height}}{$\m@th\rightarrow$}}}$\m@th\rightarrow$}
\def\smallunderbrace#1{\mathop{\vtop{\m@th\ialign{##\crcr$\hfil\displaystyle{#1}\hfil$\crcr\noalign{\kern3\p@\nointerlineskip}\tiny\upbracefill\crcr\noalign{\kern3\p@}}}}\limits}
\makeatother

%% file: DAHAs_and_Character_Varieties.bbl
\begin{thebibliography}{EGGO07}

\bibitem[AB83]{atiyahbott}
M.~Atiyah and R.~Bott.
\newblock {The Yang--Mills equations over Riemann surfaces}.
\newblock {\em Philos. Trans. R. Soc. A}, \textbf{308}:523--615, 1983.
\newblock \href {https://doi.org/10.1098/rsta.1983.0017}
  {\path{doi:10.1098/rsta.1983.0017}}.

\bibitem[ACL21]{argyreschalykhlu21}
P.~C. Argyres, O.~Chalykh, and Y.~L\"u.
\newblock {Inozemtsev system as Seiberg--Witten integrable system}.
\newblock {\em J. High Energ. Phys.}, \textbf{2021}(5):1--28, 2021.
\newblock \href {https://doi.org/10.1007/JHEP05(2021)051}
  {\path{doi:10.1007/JHEP05(2021)051}}.

\bibitem[AR02]{avanrollet}
J.~Avan and G.~Rollet.
\newblock {Structures in $BC_N$ Ruijsenaars--Schneider models}.
\newblock {\em J. Math. Phys.}, \textbf{43}(1):403--416, 2002.
\newblock \href {https://doi.org/10.1063/1.1423766}
  {\path{doi:10.1063/1.1423766}}.

\bibitem[AR21]{arthamonovreshetikhin}
S.~Arthamonov and N.~Reshetikhin.
\newblock {Superintegrable systems on moduli spaces of flat connections}.
\newblock {\em Commun. Math. Phys}, \textbf{386}:1337--1381, 2021.
\newblock \href {https://doi.org/10.1007/s00220-021-04128-5}
  {\path{doi:10.1007/s00220-021-04128-5}}.

\bibitem[BCE08]{berestchalykheshmatov}
Y.~Berest, O.~Chalykh, and F.~Eshmatov.
\newblock {Recollement of deformed preprojective algebras and the
  Calogero--Moser correspondence}.
\newblock {\em Mosc. Math. J.}, \textbf{8}(1):21--37, 2008.
\newblock \href {https://doi.org/10.17323/1609-4514-2008-8-1-21-37}
  {\path{doi:10.17323/1609-4514-2008-8-1-21-37}}.

\bibitem[BCR18]{bertolacafassorubtsov}
M.~Bertola, M.~Cafasso, and V.~Rubtsov.
\newblock {Noncommutative Painlev\' e equations and systems of Calogero type}.
\newblock {\em Commun. Math. Phys.}, \textbf{363}:503--530, 2018.
\newblock \href {https://doi.org/10.1007/s00220-018-3210-0}
  {\path{doi:10.1007/s00220-018-3210-0}}.

\bibitem[CB04]{crawley-boevey}
W.~Crawley-Boevey.
\newblock {Indecomposable parabolic bundles}.
\newblock {\em Publ. Math., Inst. Hautes \'Etud. Sci.}, \textbf{100}:171--207,
  2004.
\newblock \href {https://doi.org/10.1007/s10240-004-0025-7}
  {\path{doi:10.1007/s10240-004-0025-7}}.

\bibitem[CBS06]{crawley-boeveyshaw}
W.~Crawley-Boevey and P.~Shaw.
\newblock {Multiplicative preprojective algebras, middle convolution and the
  Deligne--Simpson problem}.
\newblock {\em Adv. Math.}, \textbf{201}:180--208, 2006.
\newblock \href {https://doi.org/10.1016/j.aim.2005.02.003}
  {\path{doi:10.1016/j.aim.2005.02.003}}.

\bibitem[CF20]{chalykhfairon}
O.~Chalykh and M.~Fairon.
\newblock {On the Hamiltonian formulation of the trigonometric spin
  Ruijsenaars--Schneider system}.
\newblock {\em Lett. Math. Phys.}, \textbf{110}:2893--2940, 2020.
\newblock \href {https://doi.org/10.1007/s11005-020-01320-x}
  {\path{doi:10.1007/s11005-020-01320-x}}.

\bibitem[Cha19]{chalykh}
O.~Chalykh.
\newblock {Quantum Lax pairs via Dunkl and Cherednik operators}.
\newblock {\em Commun. Math. Phys.}, \textbf{369}(1):261--316, 2019.
\newblock \href {https://doi.org/10.1007/s00220-019-03289-8}
  {\path{doi:10.1007/s00220-019-03289-8}}.

\bibitem[Che92]{cherednik92}
I.~Cherednik.
\newblock {Double affine Hecke algebras, Knizhnik--Zamolodchikov equations, and
  Macdonald's operators}.
\newblock {\em IMRN}, \textbf{1992}(9):171--180, 1992.
\newblock \href {https://doi.org/10.1155/S1073792892000199}
  {\path{doi:10.1155/S1073792892000199}}.

\bibitem[Che95a]{cherednikJan95}
I.~Cherednik.
\newblock {Double affine Hecke algebras and Macdonald’s conjectures}.
\newblock {\em Ann. Math.}, \textbf{141}(1):191--216, 1995.
\newblock \href {https://doi.org/10.2307/2118632} {\path{doi:10.2307/2118632}}.

\bibitem[Che95b]{cherednikDec95}
I.~Cherednik.
\newblock {Macdonald's evaluation conjectures and difference Fourier
  transform}.
\newblock {\em Invent. Math.}, \textbf{122}(1):119--145, 1995.
\newblock \href {https://doi.org/10.1007/bf01231441}
  {\path{doi:10.1007/bf01231441}}.

\bibitem[EG02]{etingofginzburg}
P.~Etingof and V.~Ginzburg.
\newblock {Symplectic reflection algebras, Calogero--Moser space, and deformed
  Harish--Chandra homomorphism}.
\newblock {\em Invent. Math.}, \textbf{147}:243--348, 2002.
\newblock \href {https://doi.org/10.1007/s002220100171}
  {\path{doi:10.1007/s002220100171}}.

\bibitem[EGGO07]{etingofganginzburgoblomkov}
P.~Etingof, W.~L. Gan, V.~Ginzburg, and A.~Oblomkov.
\newblock {Harish--Chandra homomorphisms and symplectic reflection algebras for
  wreath-products}.
\newblock {\em Publ. Math. IH\'ES}, \textbf{105}:91--155, 2007.
\newblock \href {https://doi.org/10.1007/s10240-007-0005-9}
  {\path{doi:10.1007/s10240-007-0005-9}}.

\bibitem[EGO06]{etingofganoblomkov}
P.~Etingof, W.~L. Gan, and A.~Oblomkov.
\newblock {Generalized double affine Hecke algebras of higher rank}.
\newblock {\em J. Reine Angew. Math.}, \textbf{600}:177--201, 2006.
\newblock \href {https://doi.org/10.1515/CRELLE.2006.091}
  {\path{doi:10.1515/CRELLE.2006.091}}.

\bibitem[ELOR08]{etingofloktevoblomkovrybnikov}
P.~Etingof, S.~Loktev, A.~Oblomkov, and L.~Rybnikov.
\newblock {A Lie-theoretic construction of spherical symplectic reflection
  algebras}.
\newblock {\em Transform. Groups}, \textbf{13}(3):541--556, 2008.
\newblock \href {https://doi.org/10.1007/s00031-008-9035-8}
  {\path{doi:10.1007/s00031-008-9035-8}}.

\bibitem[EOR07]{etingofoblomkovrains}
P.~Etingof, A.~Oblomkov, and E.~Rains.
\newblock {Generalized double affine Hecke algebras of rank $1$ and quantized
  del Pezzo surfaces}.
\newblock {\em Adv. Math.}, \textbf{212}:749--796, 2007.
\newblock \href {https://doi.org/10.1016/j.aim.2006.11.008}
  {\path{doi:10.1016/j.aim.2006.11.008}}.

\bibitem[FGNR00]{fockgorskynekrasovrubtsov}
V.~Fock, A.~Gorsky, N.~Nekrasov, and V.~Rubtsov.
\newblock {Duality in integrable systems and gauge theories}.
\newblock {\em J. High Energy Phys.}, \textbf{2000}(7), 2000.
\newblock \href {https://doi.org/10.1088/1126-6708/2000/07/028}
  {\path{doi:10.1088/1126-6708/2000/07/028}}.

\bibitem[FK11]{feherklimcik}
L.~Feh\'er and C.~Klim\v{c}\'ik.
\newblock {Poisson--Lie interpretation of trigonometric Ruijsenaars duality}.
\newblock {\em Commum. Math. Phys.}, \textbf{301}:55--104, 2011.
\newblock \href {https://doi.org/10.1007/s00220-010-1140-6}
  {\path{doi:10.1007/s00220-010-1140-6}}.

\bibitem[FM17]{fehermarshall17}
L.~Feh\'er and I.~Marshall.
\newblock {The action-angle dual of an integrable Hamiltonian system of
  Ruijsenaars--Schneider--van Diejen type}.
\newblock {\em J. Phys. A: Math. Theor.}, \textbf{50}(31):314004, 2017.
\newblock \href {https://doi.org/10.1088/1751-8121/aa7934}
  {\path{doi:10.1088/1751-8121/aa7934}}.

\bibitem[FM19]{fehermarshall19}
L.~Feh\'er and I.~Marshall.
\newblock {Global description of action-angle duality for a Poisson--Lie
  deformation of the trigonometric $\mathrm{BC}_n$ Sutherland system}.
\newblock {\em Ann. Henri Poincar\'e}, \textbf{20}:1217--1262, 2019.
\newblock \href {https://doi.org/0.1007/s00023-019-00782-7}
  {\path{doi:0.1007/s00023-019-00782-7}}.

\bibitem[FR99]{fockrosly}
V.~Fock and A.~Rosly.
\newblock {Poisson structure on moduli of flat connections on Riemann surfaces
  and the $r$-matrix}.
\newblock {\em Am. Math. Soc. Transl.}, \textbf{191}:67--86, 1999.
\newblock \href {https://doi.org/10.1090/trans2/191}
  {\path{doi:10.1090/trans2/191}}.

\bibitem[GN95]{gorskynekrasov}
A.~Gorsky and N.~Nekrasov.
\newblock {Relativistic Calogero--Moser model as gauged WZW theory}.
\newblock {\em Nucl. Phys. B}, \textbf{436}:582--608, 1995.
\newblock \href {https://doi.org/10.1016/0550-3213(94)00499-5}
  {\path{doi:10.1016/0550-3213(94)00499-5}}.

\bibitem[Gol84]{goldman84}
W.~M. Goldman.
\newblock {The symplectic nature of fundamental groups of surfaces}.
\newblock {\em Adv. Math.}, \textbf{54}(2):200--225, 1984.
\newblock \href {https://doi.org/10.1016/0001-8708(84)90040-9}
  {\path{doi:10.1016/0001-8708(84)90040-9}}.

\bibitem[Gol86]{goldman86}
W.~M. Goldman.
\newblock {Invariant functions on Lie groups and Hamiltonian flows of surface
  group representations}.
\newblock {\em Invent. Math.}, \textbf{85}:263--302, 1986.
\newblock \href {https://doi.org/10.1007/BF01389091}
  {\path{doi:10.1007/BF01389091}}.

\bibitem[Gro14]{groechenig}
M.~Groechenig.
\newblock {Hilbert schemes as moduli of Higgs bundles and local systems}.
\newblock {\em IMRN}, \textbf{2014}(23):6523--6575, 2014.
\newblock \href {https://doi.org/10.1093/imrn/rnt167}
  {\path{doi:10.1093/imrn/rnt167}}.

\bibitem[HLRV11]{hauselletellierrodriguez-villegas11}
T.~Hausel, E.~Letellier, and F.~Rodriguez-Villegas.
\newblock {Arithmetic harmonic analysis on character and quiver varieties}.
\newblock {\em Duke Math. J.}, \textbf{160}(2):323--400, 2011.
\newblock \href {https://doi.org/10.1215/00127094-1444258}
  {\path{doi:10.1215/00127094-1444258}}.

\bibitem[HLRV13]{hauselletellierrodriguez-villegas13}
T.~Hausel, E.~Letellier, and F.~Rodriguez-Villegas.
\newblock {Arithmetic harmonic analysis on character and quiver varieties II}.
\newblock {\em Adv. Math.}, \textbf{234}:85--128, 2013.
\newblock \href {https://doi.org/10.1016/j.aim.2012.10.009}
  {\path{doi:10.1016/j.aim.2012.10.009}}.

\bibitem[Jor14]{jordan}
D.~Jordan.
\newblock {Quantized multiplicative quiver varieties}.
\newblock {\em Adv. Math.}, \textbf{250}(15):420--466, 2014.
\newblock \href {https://doi.org/10.1016/j.aim.2013.09.010}
  {\path{doi:10.1016/j.aim.2013.09.010}}.

\bibitem[Kaw15]{kawakami}
H.~Kawakami.
\newblock {Matrix Painlev\'e systems}.
\newblock {\em J. Math. Phys.}, \textbf{56}:033503, 2015.
\newblock \href {https://doi.org/10.1063/1.4914369}
  {\path{doi:10.1063/1.4914369}}.

\bibitem[KKS78]{kazhdankostantsternberg}
D.~Kazhdan, B.~Kostant, and S.~Sternberg.
\newblock {Hamiltonian group actions and dynamical systems of Calogero type}.
\newblock {\em Comm. Pure Appl. Math.}, \textbf{31}:481--507, 1978.
\newblock \href {https://doi.org/10.1002/cpa.3160310405}
  {\path{doi:10.1002/cpa.3160310405}}.

\bibitem[Mac03]{macdonald}
I.~G. Macdonald.
\newblock {\em Affine Hecke Algebras and Orthogonal Polynomials}.
\newblock Cambridge Tracts in Mathematics. Cambridge University Press, 2003.
\newblock {ISBN:} 978-0-521-82472-9.

\bibitem[MT14]{massuyeauturaev}
G.~Massuyeau and V.~Turaev.
\newblock {Quasi-Poisson structures on representation spaces of surfaces}.
\newblock {\em IMRN}, \textbf{2014}(1):1--64, 2014.
\newblock \href {https://doi.org/10.1093/imrn/rns215}
  {\path{doi:10.1093/imrn/rns215}}.

\bibitem[Nou95]{noumi}
M.~Noumi.
\newblock {Macdonald--Koornwinder polynomials and affine Hecke rings
  (\textit{in Japanese})}.
\newblock {\em {\normalfont In:} Various Aspects of Hypergeometric Functions
  {\normalfont(Kyoto, 1994)}}, \textbf{919}:44--55, 1995.

\bibitem[Obl04a]{oblomkovJun04}
A.~Oblomkov.
\newblock {Double affine Hecke algebras and Calogero--Moser spaces}.
\newblock {\em Represent. Theory}, \textbf{8}:243--266, 2004.
\newblock \href {https://doi.org/10.1090/S1088-4165-04-00246-8}
  {\path{doi:10.1090/S1088-4165-04-00246-8}}.

\bibitem[Obl04b]{oblomkovJan04}
A.~Oblomkov.
\newblock {Double affine Hecke algebras of rank $1$ and affine cubic surfaces}.
\newblock {\em IMRN}, \textbf{2004}:877--912, 2004.
\newblock \href {https://doi.org/10.1155/S1073792804133072}
  {\path{doi:10.1155/S1073792804133072}}.

\bibitem[PG17]{pusztaigorbe}
B.~G. Pusztai and T.~G\"orbe.
\newblock {Lax representation of the hyperbolic van Diejen dynamics with two
  coupling parameters}.
\newblock {\em Comm. Math. Phys.}, \textbf{354}:829--864, 2017.
\newblock \href {https://doi.org/10.1007/s00220-017-2935-5}
  {\path{doi:10.1007/s00220-017-2935-5}}.

\bibitem[Rai16]{rains16}
E.~M. Rains.
\newblock {The noncommutative geometry of elliptic difference equations}, 2016.
\newblock \href {https://arxiv.org/abs/1607.08876} {\path{arXiv:1607.08876}}.

\bibitem[Rai19]{rains19}
E.~M. Rains.
\newblock {The birational geometry of noncommutative surfaces}, 2019.
\newblock \href {https://arxiv.org/abs/1907.11301} {\path{arXiv:1907.11301}}.

\bibitem[Rai20]{rains20}
E.~M. Rains.
\newblock {Elliptic double affine Hecke algebras}.
\newblock {\em SIGMA}, \textbf{16}(111), 2020.
\newblock \href {https://doi.org/10.3842/SIGMA.2020.111}
  {\path{doi:10.3842/SIGMA.2020.111}}.

\bibitem[Rui88]{ruijsenaars}
S.~N.~M. Ruijsenaars.
\newblock {Action-angle maps and scattering theory for some finite-dimensional
  integrable systems}.
\newblock {\em Comm. Math. Phys.}, \textbf{115}(1):127--165, 1988.
\newblock \href {https://doi.org/10.1007/BF01238855}
  {\path{doi:10.1007/BF01238855}}.

\bibitem[Sah99]{sahi}
S.~Sahi.
\newblock {Nonsymmetric Koornwinder polynomials and duality}.
\newblock {\em Ann. Math.}, \textbf{150}(1):267--282, 1999.
\newblock \href {https://doi.org/10.2307/121102} {\path{doi:10.2307/121102}}.

\bibitem[Sto11]{stokman}
J.~Stokman.
\newblock {Quantum affine Knizhnik--Zamolodchikov equations and quantum
  spherical functions, I}.
\newblock {\em IMRN}, \textbf{2011}(5):1023--1090, 2011.
\newblock \href {https://doi.org/10.1093/imrn/rnq094}
  {\path{doi:10.1093/imrn/rnq094}}.

\bibitem[vD95]{vandiejen}
J.~F. van Diejen.
\newblock {Commuting difference operators with polynomial eigenfunctions}.
\newblock {\em Compos. Math.}, \textbf{95}:183--233, 1995.
\newblock URL: \url{http://eudml.org/doc/90347}.

\bibitem[VdB08]{vandenbergh}
M.~Van~den Bergh.
\newblock {Double Poisson algebras}.
\newblock {\em Trans. Amer. Math. Soc.}, \textbf{360}(11):5711--5769, 2008.
\newblock \href {https://doi.org/10.1090/S0002-9947-08-04518-2}
  {\path{doi:10.1090/S0002-9947-08-04518-2}}.

\bibitem[vDEZ18]{vandiejenemsizzurrian}
J.~F. van Diejen, E.~Emsiz, and I.~N. Zurri\'an.
\newblock {Completeness of the Bethe ansatz for an open $q$-boson system with
  integrable boundary interactions}.
\newblock {\em Ann. Henri Poincar\'e}, \textbf{19}(5):1349--1384, 2018.
\newblock \href {https://doi.org/10.1007/s00023-018-0658-6}
  {\path{doi:10.1007/s00023-018-0658-6}}.

\bibitem[Wen24]{wen}
J.~J. Wen.
\newblock {Harish-Chandra isomorphism for cyclotomic double affine Hecke
  algebras}, 2024.
\newblock \href {https://arxiv.org/abs/2407.07679} {\path{arXiv:2407.07679}}.

\bibitem[Wil98]{wilson}
G.~Wilson.
\newblock {Collisions of Calogero--Moser particles and an adelic Grassmannian
  (with an appendix by I. G. Macdonald)}.
\newblock {\em Invent. Math.}, \textbf{133}:1--41, 1998.
\newblock \href {https://doi.org/10.1007/s002220050237}
  {\path{doi:10.1007/s002220050237}}.

\end{thebibliography}
